\documentclass[10pt,reqno]{amsart}
\usepackage[top=30truemm,bottom=30truemm,left=25truemm,right=25truemm]{geometry}

\usepackage{amssymb, amsmath, amsthm, amsfonts, latexsym, mathtools, amsfonts}
\usepackage{graphicx}
\usepackage{xcolor}
\usepackage{ytableau}
\usepackage{tikz, comment}
\usepackage{amscd}
\usepackage{dsfont}
\usepackage{arydshln}
\usepackage[normalem]{ulem}
\usepackage{mathtools}
\usepackage{thm-restate}

\usepackage{hyperref}
\usepackage[nameinlink]{cleveref}
\hypersetup{
setpagesize=false,
bookmarksnumbered=true,%
bookmarksopen=true,%
colorlinks=true,%
linkcolor=orange,
citecolor=blue,
}

\crefname{equation}{}{}

\newtheorem{thm}{Theorem}[section]
\crefname{thm}{Theorem}{Theorems}

\crefname{prop}{Proposition}{Propositions}
\newtheorem{lem}[thm]{Lemma}
\crefname{lem}{Lemma}{Lemmas}
\newtheorem{cor}[thm]{Corollary}
\crefname{cor}{Corollary}{Corollarys}
\newtheorem{conj}[thm]{Conjecture}
\crefname{conj}{Conjecture}{Conjectures}

\theoremstyle{definition}
\newtheorem*{definition}{Definition}

\newtheorem*{remark}{Remark}

\crefname{section}{Section}{Sections}

\NewDocumentEnvironment{restatedthm}{m O{Restated}}
{
  \par\addvspace{\topsep}
  \hypersetup{linkcolor=black}
  \noindent
  \textbf{\cref{#1} (#2).}
  \hypersetup{linkcolor=orange}~
  \itshape
}
{
  \par\addvspace{\topsep}
}

\newcommand{\til}{\widetilde}

\newcommand{\vphi}{\varphi}
\newcommand{\ep}{\varepsilon}

\newcommand{\Img}{\operatorname{Im}}

\newcommand{\Real}{\operatorname{Re}}

\newcommand{\ceq}{\coloneq}

\newcommand{\sgn}{\operatorname{sgn}}

\newcommand{\CC}{\mathbb{C}}
\newcommand{\RR}{\mathbb{R}}
\newcommand{\QQ}{\mathbb{Q}}
\newcommand{\ZZ}{\mathbb{Z}}

\newcommand{\HH}{\mathbb{H}}
\newcommand{\GG}{\mathbb{G}}

\numberwithin{equation}{section}

\begin{document}

\title{Zeros of quasimodular forms defined by iterated sums}
\author{K. Kina}
\address{Graduate School of Mathematics, Kyushu University, Motooka 744, Nishi-ku, Fukuoka 819-0395 Japan}
\email{kkina@math.kyushu-u.ac.jp}
\author{G. Shin}
\address{KNU Research Institute of Mathematical Sciences, Kangwon National University, Chuncheon 24341, Republic of Korea}
\email{gshin@kangwon.ac.kr}

\begin{abstract}
We study the zeros of the quasimodular forms $G_{\{2\}^n}$ defined by iterated sums. We first show that, for every $n>0$, $G_{\{2\}^n}$ has exactly $n$ simple zeros on each of the vertical half-lines $\Real(\tau)=0$ and $\Real(\tau)=1/2$, and that the zeros for consecutive values of $n$ satisfy an interlacing property. The proof is based on an expression of $G_{\{2\}^n}$ in terms of the $n$-th derivative of $\eta^3$ and on the theory of bell-shaped functions, rather than on Rankin--Swinnerton-Dyer method. We also determine the asymptotic behavior of these zeros as $n\to\infty$. In addition, we prove that all zeros of $G_{\{2\}^n}$ are simple and that $G_{\{2\}^n}$ has infinitely many $SL_2(\ZZ)$-inequivalent zeros. We further show that quasimodular forms of maximal depth have no zeros at CM points. In particular, none of the zeros of $G_{\{2\}^n}$ are CM points. Finally, in the special case $G_{2,2}$, we show that each Ford circle contains exactly two distinct simple zeros.
\end{abstract}

\maketitle


\section{Introduction}

The study of zeros of quasimodular forms has largely developed in two directions: the distribution of zeros and the transcendence of zeros. In this paper, we study the distribution and transcendence of zeros of quasimodular forms defined by iterated sums.

\subsection{The zeros of Eisenstein series and RSD method}
Rankin and Swinnerton-Dyer proved in \cite{RS70} that, for even $k\geq 4$, all the zeros of the Eisenstein series
$$\GG_{k}(\tau) \ceq \underset{(m,n)\neq(0,0)}{\sum_{m\in\ZZ}\ \sum_{n\in\ZZ}} \frac{1}{(m\tau+n)^k},$$
in the standard fundamental domain
$$\mathcal{F}\ceq \{\tau\in\HH\mid -1/2\leq \Real \tau\leq 0,~|\tau|\geq 1\}
\cup \{\tau\in\HH\mid 0< \Real \tau< 1/2,~|\tau|> 1\}$$
lie on the arc of the unit circle given by $\{e^{it} \mid \pi/2\leq t\leq 2\pi/3\}$. Here, $\HH$ is the upper-half plane. We briefly recall their argument. Define the function $F_k:\RR\to\RR~;~F_k(t)=e^{ikt/2}\GG_k(e^{it})$. Then, for each integer $m\in[k/4, k/3]$, one has $(-1)^m F_k(2m\pi/k)>0$. The resulting sign changes give a lower bound for the number of zeros of $\GG_k$ on the arc. Combining this with the valence formula, which states that the weighted sum of the orders of the zeros in the fundamental domain is $k/12$, one obtains their result. We refer to this argument as the RSD method.

Miezaki, Nozaki, and Shigezumi in \cite{MNS07} applied RSD method to the study of zeros of Eisenstein series for the Fricke groups $\Gamma_0^*(2)$ and $\Gamma_0^*(3)$. Moreover, Gun and Oesterl\'{e} in \cite{GO22} adapted the RSD method to prove that, for even $k\geq 4$, all the zeros of $\GG_k'$ in $\mathcal{F}$ lie on the line $\Real(\tau)=-1/2$. Sugibayashi in \cite{Sug26} obtained an analogous result for the derivatives of Eisenstein series for $\Gamma_0^*(2)$.

As another interesting development, van Ittersum and Ringeling in \cite{IR25} obtained a result analogous to that of \cite{GO22} for Eisenstein series of odd weight, which are not even quasimodular forms.

Our \cref{thm:num-zero} is regarded as an analogue of their result. However, a novel feature of our result is that it is obtained without using RSD method employed in their work.

\subsection{The zeros of depth one quasimodular forms}
Although relatively little is known about the zeros of quasimodular forms of arbitrary depth, several results have been obtained in the depth-one case.

El Basraoui and Sebbar \cite{BS10} studied the zeros of $\GG_2$ and proved that it has infinitely many zeros in the half-strip $\{\tau\in\HH\mid -1/2\leq \Real(\tau)<1/2\}$, all of which are simple and pairwise $SL_2(\ZZ)$-inequivalent. Subsequently, Imamoglu, Jermann, and T\'{o}th \cite{IJT14} and, independently, Wood and Young \cite{WY14} refined these results by showing that each Ford circle contains exactly one simple zero of $\GG_2$, that $\GG_2$ has no other zeros in the upper half-plane, and by obtaining more precise estimates for the location of the zero in each Ford circle. For further related results, see also \cite{SS12}, \cite{BG12}, and \cite{Meh13}.

Valence formulas for derivatives of Eisenstein series, which provide typical examples of depth-one quasimodular forms, were subsequently established in \cite{GO22,Sug26}. Van Ittersum and Ringeling \cite{IR24} established a valence formula for arbitrary quasimodular forms of depth one and also considered quasimodular forms of arbitrary depth. More recently, Im and Lee \cite{IL26} studied real zeros of depth-one quasimodular forms and obtained several results extending those of \cite{GO22}.

\subsection{quasimodular forms and multiple Eisenstein series}
quasimodular forms were introduced by Kaneko and Zagier \cite{KZ95}. The ring of quasimodular forms for $SL_2(\ZZ)$ is given by $\widetilde{M}=\CC[G_2,G_4,G_6]$. Moreover, $G_2,G_4,G_6$ are algebraically independent over $\CC$. Hence, the natural homomorphism
$$\CC[X,Y,Z]\longrightarrow \til{M},
\qquad
P(X,Y,Z)\longmapsto P(G_2,G_4,G_6),
$$
is an isomorphism. In particular, every quasimodular form admits a unique expression as a polynomial in $G_2,G_4,G_6$. Under this isomorphism, differentiation with respect to $G_2$ is defined as differentiation with respect to $X$ on $\CC[X,Y,Z]$. We say that a quasimodular form $f$ has weight $k$ if $f=P(G_2,G_4,G_6)$, where $P(X,Y,Z)\in\CC[X,Y,Z]$ is weighted homogeneous of weight $k$, that is, $P(t^2X,t^4Y,t^6Z)=t^kP(X,Y,Z)$. The depth of $f$ is the degree of $P$ with respect to $X$.

Multiple Eisenstein series were introduced in \cite{GKZ06} as a
generalization of the ordinary Eisenstein series $G_k$ via iterated
sums. More precisely, let $r\in\mathbb{Z}_{>0}$,
$k_1,\ldots,k_r\in\mathbb{Z}_{\geq 2}$, and $\tau\in\mathbb{H}$.
Set $\mathbb{Z}_N:=\{n\in\mathbb{Z}\mid |n|<N\}$.
We define an order $\prec$ on the lattice
$\mathbb{Z}\tau+\mathbb{Z}$ by
\begin{align*}
0\prec m\tau+n &\overset{\text{def}}{\Longleftrightarrow}
\left\{\begin{array}{c}
m>0
\\
\text{or}
\\
m=0 \quad \text{and}\quad n>0
\end{array}\right. ,
\intertext{and}
m_1\tau+n_1\prec m_2\tau+n_2 &\overset{\text{def}}{\Longleftrightarrow}
0\prec (m_2-m_1)\tau+(n_2-n_1).
\end{align*}
The multiple Eisenstein series of depth $r$ and index
$(k_1,\ldots,k_r)$ is given by the iterated sum
$$G_{k_1,\ldots,k_r}(\tau) \ceq
  \lim_{M\to\infty}\lim_{N\to\infty}
  \sum_{\substack{
    w_s\in\mathbb{Z}_M\tau+\mathbb{Z}_N\\
    0\prec w_1\prec\cdots\prec w_r
  }}
  \frac{1}{w_1^{k_1}\cdots w_r^{k_r}}.$$
The space spanned by these functions has been studied from various perspectives; see, for example, \cite{Bac26} and the references therein. In this paper, we consider only multiple Eisenstein series $G_{\{2\}^n}$ whose indices have all entries equal to $2$. Accordingly, we can regard the following formula, due to Hoffman and Ihara \cite{HI17}, as the definition of $G_{\{k\}^n}$.
\begin{equation}\label{eq:def-G}
\sum_{n=0}^{\infty} G_{\{k\}^n}(\tau)X^{n}
= \exp\left(\sum_{m=1}^{\infty}(-1)^{m-1}\frac{G_{km}(\tau)}{m}X^{m}\right).
\end{equation}
Here, for an integer $k\geq 2$,
$$G_{k}(\tau) \ceq \zeta(k) + \frac{(2\pi i)^{k}}{(k-1)!}\sum_{n=1}^{\infty}\sigma_{k-1}(n)q^n\ ,
\quad \text{where}\quad
\sigma_{k-1}(n) = \sum_{\substack{d>0\\d|n}} d^{k-1}$$
and $q=e^{2\pi i\tau}$. Note that $2G_k=\GG_k$ when $k$ is even. For example, we have
$$G_{\{2\}^0}
=1\ ,\quad  G_{2,2} = \frac{1}{2}(G_2^2 - G_4)\ ,\quad  G_{2,2,2} = \frac{1}{6}G_2^3 - \frac{1}{2}G_2G_4 + \frac{1}{3}G_6.$$
It follows from \cref{eq:def-G} that
\begin{equation}\label{eq:der-G}
\frac{d}{dG_2}G_{\{2\}^n} = G_{\{2\}^{n-1}}
\end{equation}
and hence $G_{\{2\}^n}$ is a quasimodular form of weight $2n$ and depth $n$. Moreover,
$$\QQ[G_2,G_4,G_6]=\QQ[G_2,G_{2,2},G_{2,2,2}].$$

\subsection{Transcendence of quasimodular forms}
Transcendence questions for quasimodular forms have been studied from various perspectives. A classical Schneider theorem \cite{Sch37} states that if $\tau\in\HH$ is algebraic but not a CM point, then $j(\tau)$ is transcendental. In a different direction, Nesterenko in \cite{Nes96} established a remarkable algebraic independence result for Eisenstein series. These results have become fundamental tools in the transcendence theory of quasimodular forms. These ideas have been further developed to study special values and zeros of quasimodular forms.

\subsection{Summary of our results}
In this article, we are interested in the zeros of the quasimodular form $G_{\{2\}^n}$ of weight $2n$ and depth $n$. By \cref{lem:eta-G}, which will be stated later, studying the zeros of $G_{\{2\}^n}$ is equivalent to studying the zeros of the $n$-th derivative of $\eta^3$. It is easy to see from the identities
$$G_{\{2\}^n}(\tau)=G_{\{2\}^n}(\tau+1)
\quad\text{and}\quad
G_{\{2\}^n}(-\overline{\tau})=\overline{G_{\{2\}^n}(\tau)}$$
that the zeros of $G_{\{2\}^n}$ are periodic with period $1$ and symmetric with respect to the imaginary axis. Our results are as follows.

\begin{thm}\label{thm:num-zero}
For each integer $n>0$, the function $G_{\{2\}^n}$ has exactly $n$ simple zeros lying on $\{iy\mid y\in \RR_{>0}\}$ and $\{1/2+iy\mid y\in \RR_{>0}\}$, respectively.
\end{thm}

The proof of the result does not use RSD method, which has been highly effective in previous studies. This is a notable difference from related works. A key ingredient in our proof is a result of Kwa\'{s}nicki \cite{Kwa20} concerning bell-shaped functions.

Interlacing of zeros arises naturally in the study of Eisenstein series. Nozaki \cite{Noz08} proved that the zeros of $E_k$ and $E_{k+12}$ interlace on the lower arc of the standard fundamental domain, and this result was later generalized by Griffin et al. \cite{GKPVXZ21} to pairs $E_k$ and $E_{k+a}$. Throughout this article, we say that two finite subsets $A,B\subset\RR$ have the interlacing property if $\bigl||A|-|B|\bigr|=1$ and, when their elements are arranged in increasing order, the elements of $A$ and $B$ occur alternately. The following corollary, which is established in the course of the proof of the preceding theorem, shows that the same phenomenon also occurs in our setting.

\begin{cor}
Set $S_0^{(n)}\ceq\{y\in\RR_{>0}:G_{\{2\}^n}(iy)=0\}$ and $S_{1/2}^{(n)}\ceq\{y\in\RR_{>0}:G_{\{2\}^n}(1/2+iy)=0\}$. Then for all $n\ge1$ and $*\in\{0,1/2\}$, the sets $S_*^{(n)}$ and $S_*^{(n+1)}$ have the interlacing property.
\end{cor}

Let $iy^{(0,n)}_1,\dots,iy^{(0,n)}_n$ and $\frac{1}{2}+iy^{(1/2,n)}_1,\dots,\frac{1}{2}+iy^{(1/2,n)}_n$ be the zeros appearing in \cref{thm:num-zero}, respectively, ordered so that $y^{(*,n)}_1>\dots>y^{(*,n)}_n$ for $*\in\{0,1/2\}$. As we explain in \Cref{sec:bell-shaped}, for each fixed $r$,  Corollary 1.6 in \cite{KS22} implies that the sequence $\{y^{(*,n)}_r/n\}_{n\geq r}$ is bounded and therefore has an accumulation point. In fact, we show that this sequence converges to the following value.

\begin{thm}\label{thm:lim-zero}
Set $\kappa=1$ if $*=0$, and $\kappa=2$ if $*=1/2$. For each fixed integer $r>0$, we have
$$\lim_{n\to \infty} \frac{y_{r}^{(*,n)}}{n}= \frac{1}{\pi \kappa r}\log\left(\frac{2\kappa r+1}{2\kappa r-1}\right).$$
More precisely, we have
$$\lim_{n\to \infty} \left( y_{r}^{(*,n)} - \left(n+\frac{1}{2}\right)\frac{1}{\pi \kappa r}\log\left(\frac{2\kappa r+1}{2\kappa r-1}\right)\right) = 0.$$
\end{thm}

More generally, we have the following.
\begin{thm}\label{thm:lim-zero-gen}
For any integers $r>0$ and $j$, there exists a sequence $\{z_{r,j}^{(n)}\}_{n\geq r}$ of zeros of $G_{\{2\}^n}$, such that
$$\lim_{n\to \infty} \left( z_{r,j}^{(n)} - \left(\frac{j}{r} + i\left(n+\frac{1}{2}\right)\frac{1}{\pi r}\log\left(\frac{2r+1}{2r-1}\right)\right)\right) = 0.$$
\end{thm}

The following two theorems generalize the results of \cite{BS10}.
\begin{thm}\label{thm:simple-zero}
For each integer $n>0$, all zeros of $G_{\{2\}^n}$ are simple.
\end{thm}

\begin{thm}\label{thm:equiv-zero}
Let $f$ be a nonzero quasimodular form of weight $k$ and depth $k/2$. Suppose that $f$ has a simple zero. Then $f$ has infinitely many $SL_2(\ZZ)$-inequivalent zeros in the half-strip. In particular, the same holds for $G_{\{2\}^n}$ for every integer $n>0$.
\end{thm}

The following conjecture follows immediately from the transformation law of $G_2$ when $n=1$ (See \cite[Proposition 3.3]{BS10}). However, for general $n$, it remains unproved even in the case $n=2$. 

\begin{conj}
For each integer $n>0$, all zeros of $G_{\{2\}^n}$ in the half-strip are pairwise $SL_2(\ZZ)$-inequivalent in the half-strip.
\end{conj}

As a consequence of \cite[Theorem~1]{GMR11}, every zero of $G_k$, for even $k \geq 4$, is either a CM point or a transcendental number. On the other hand, \cite{Kum21} conjectured that all zeros of $G_2$ are transcendental. As a simple application of either \cite{Chu76} or \cite{Nes96}, we obtain the following result. However, it remains unknown whether $G_{\{2\}^n}$ has any algebraic zeros other than CM points; see the remark in \cref{subsec:trans}.

\begin{thm}\label{thm:trans-zero}
Let $f\in\overline{\QQ}[G_2,G_4,G_6]$ be a nonzero quasimodular form of weight $k$ and depth $k/2$. Then none of the zeros of $f$ in $\HH$ is a CM point. In particular, for every $n>0$, none of the zeros of $G_{\{2\}^n}$ is a CM point.
\end{thm}

Furthermore, by the Gelfond--Schneider theorem \cite{Gel34,Sch35}, $\log a/\pi$ is transcendental for every algebraic real number $a>0$ with $a\neq1$. In particular, the limiting values for the zeros appearing in \cref{thm:lim-zero},
$$\frac{1}{\pi\kappa r}\log\left(\frac{2\kappa r+1}{2\kappa r-1}\right),
\qquad \kappa\in\{1,2\},$$
are transcendental (see \cref{cor:trans-lim-cor}).

\begin{figure}[htbp]
  \centering
  \begin{minipage}{0.45\textwidth}
    \centering
    \scalebox{1}[0.88]{%
      \includegraphics[width=\textwidth]{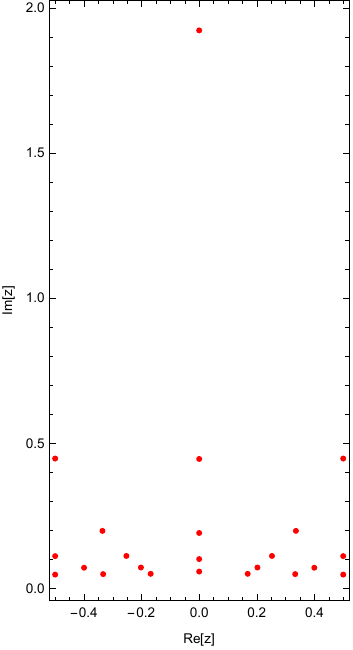}%
    }
  \end{minipage}
  \hfill
  \begin{minipage}{0.45\textwidth}
    \centering
    \scalebox{1}[0.88]{%
      \includegraphics[width=\textwidth]{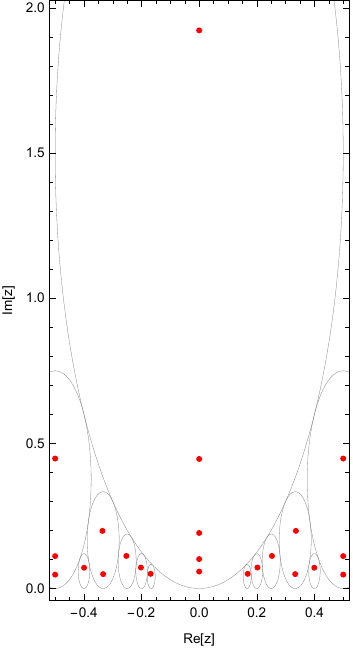}%
    }
  \end{minipage}
  \caption{Zeros of $G_{\{2\}^5}$.}
  \label{fig:zeros-5}
\end{figure}

Finally, we discuss the location of the zeros of $G_{\{2\}^n}$ within Ford circles. Recall that, for a reduced fraction $\frac{a}{c}\in\QQ$ with $c>0$, the Ford circle associated with $\frac{a}{c}$ is the circle tangent to the real axis at $\frac{a}{c}$ with radius $\frac{1}{2c^2}$.
We expect that results analogous to those of \cite{IJT14} and \cite{WY14} should also hold for $G_{\{2\}^n}$. However, in contrast to the cases considered in these works, the usual Ford circles do not appear to be large enough to contain all the zeros of $G_{\{2\}^n}$. Indeed, numerical experiments suggest that, for $n\geq 3$, $G_{\{2\}^n}$ has a zero on the positive imaginary axis whose imaginary part is greater than $1$, and hence this zero lies outside the usual Ford circle.
More precisely, we conjecture that each suitably vertically stretched Ford circle contains exactly $n$ zeros of $G_{\{2\}^n}$, and that $G_{\{2\}^n}$ has no other zeros. We further expect that the locations of these zeros within the stretched Ford circles can be estimated with substantially greater precision. At present, we have established only the following result in this direction, in the case of $G_{2,2}$.
\Cref{fig:zeros-5} illustrates the zeros of $G_{\{2\}^5}$, corresponding to the case $n=5$. In the right-hand panel, the same zeros are shown together with the Ford circles stretched vertically by a factor of $(n+1)/2$.

\begin{thm}\label{thm:fordcircle}
Each Ford circle contains exactly two distinct (and simple) zeros of $G_{2,2}$.
\end{thm}

The set $\{G_2,G_4,G_6\}$ is a standard set of generators for the ring of quasimodular forms. This choice is natural from the viewpoint of quasimodular forms as an extension of modular forms, since the ring of modular forms is generated by $G_4$ and $G_6$. On the other hand, from the viewpoint of the behavior of zeros observed above, one may regard $\{G_2,G_{2,2},G_{2,2,2}\}$ as a more natural set of generators for the ring of quasimodular forms.

\subsection{Organization of the paper}
In \cref{sec:bell-shaped}, we give a proof of \cref{thm:num-zero}. That is achieved by showing that $\eta(iy)^3$ is a bell-shaped function. In \cref{sec:asymptotic-zero}, we give proofs of \cref{thm:lim-zero,thm:lim-zero-gen}. This is obtained by explicitly examining the Fourier expansion of $\eta(\tau)^3$. In \cref{sec:global-zero}, we prove \cref{thm:equiv-zero,thm:simple-zero,thm:trans-zero}, and in \cref{sec:fordcircle}, we prove \cref{thm:fordcircle}. These results are established as applications of previous work.

\section*{Use of AI}

The authors used ChatGPT 5.6 Sol as a supplementary tool in developing some of the mathematical statements and proofs presented in this work. While several ideas, computations, and intermediate arguments were explored through discussions with the model, the mathematical results, their formulation, and the final proofs are due to the authors. All statements and proofs were independently checked and verified by the authors, who take full responsibility for the content of this paper.

\section*{Acknowledgments}
The authors would like to thank Professor Abdellah Sebbar for pointing out an issue in an earlier version of this paper. The first named author expresses his gratitude to his academic advisor Professor Masanobu Kaneko and Professor Hiroyuki Ochiai. The first named author was supported by JST SPRING, Japan Grant Number JPMJSP2136. The second named author was supported by Basic Science Research Program through the National Research Foundation of Korea (NRF) funded by the Ministry of Education (RS-2025-25415913).

\section{Bell-shaped functions and number of zeros}\label{sec:bell-shaped}

In this section, we prove \cref{thm:num-zero}. Before giving the proof, we first recall the notion of bell-shaped functions and some of their properties, following \cite{Kwa20,KS22}.

\subsection{Recall the bell-shaped functions}

We say that a function
\(f:\RR\to\RR\) changes sign exactly \(n\) times if \(n+1\)
is the maximal length \(m\) of a strictly increasing sequence $x_1<x_2<\cdots<x_m$ in \(\RR\) such that $f(x_j)f(x_{j+1})<0$ for \(j=1,\ldots,m-1\). If \(f\) is differentiable and
\(f'(x)\neq 0\) whenever \(f(x)=0\), then \(f\) changes sign exactly
\(n\) times if and only if it has exactly \(n\) zeros.

\begin{definition}
A smooth function \(f:\RR\to\RR\) is said to be
\emph{strictly bell-shaped} if, for every \(n=0,1,2,\ldots\), the
\(n\)-th derivative \(f^{(n)}\) converges to zero at \(\pm\infty\) and
\(f^{(n)}\) changes sign exactly \(n\) times. A Borel function \(f:\mathbb{R}\to\mathbb{R}\) (or a Borel measure $f$ on $\RR$) is said to be \emph{weakly bell-shaped} if, for every
\(t>0\), the function
\[
(f*\mathfrak{G}_t)(x) = \int_{-\infty}^{\infty}f(y)\mathfrak{G}_t(x-y)~dy
\]
is well-defined and strictly bell-shaped. Here, $\mathfrak{G}_t$ is the Gauss–Weierstrass kernel defined by
$$\mathfrak{G}_t(x) = \frac{1}{\sqrt{4\pi t}} e^{-x^2/(4t)}.$$
\end{definition}

\begin{thm}[Corollary 4.4 in \cite{Kwa20}]\label{thm:Kwa-base}
Every strictly bell-shaped function is weakly bell-shaped.
\end{thm}

\begin{thm}[Theorem 1.1 in \cite{Kwa20}]\label{thm:Kwa-main}
Suppose that \(f\) is a locally integrable function which converges to
zero at \(\pm\infty\), and which is decreasing near \(\infty\) and
increasing near \(-\infty\). Suppose furthermore that for
\(\xi \in \mathbb{R}\setminus\{0\}\) the Fourier transform of \(f\)
satisfies
\begin{align}\label{eq:Fourier-exp-Kwa}
\mathcal{L}f(i\xi)
&\ceq \int_{-\infty}^{\infty} e^{-i\xi x} f(x)\,dx
= \exp\left(-a\xi^2 - ib\xi + c + \int_{-\infty}^{\infty}K(i\xi,s) \varphi(s)\,ds\right),
\end{align}
where
\begin{align*}
K(i\xi,s) \ceq \frac{1}{i\xi+s} - \left(\frac{1}{s} - \frac{i\xi}{s^2} \right)
\mathbf{1}_{\mathbb{R}\setminus(-1,1)}(s)
\end{align*}
(with the former integral understood as an improper integral); here
\(a\geq 0\), \(b,c\in\mathbb{R}\), and
\(\varphi:\mathbb{R}\to\mathbb{R}\) is a function with the following
properties:
\begin{enumerate}
\item[(a)]
For every \(k\in\mathbb{Z}\), the function
\(\varphi(s)-k\) changes its sign at most once, and for \(k=0\)
this change takes place at \(s=0\): we have
$$\varphi(s)\geq 0 \quad \text{for } s>0,
\qquad
\varphi(s)\leq 0 \quad \text{for } s<0.$$
\item[(b)]
We have
$$\left(\int_{-\infty}^{-1} + \int_{1}^{\infty}\right)
\frac{|\varphi(s)|}{|s|^3}\,ds<\infty.$$
\item[(c)]
We have
$$\int_{-1}^{1}\Real\mathcal{L}f(i\xi)\,d\xi <\infty
\qquad\text{and}\qquad
\lim_{\xi\to 0} \Img\mathcal{L}f(i\xi) =0.$$
\end{enumerate}
Then \(f\) is weakly bell-shaped. If in addition \(f\) is smooth, then \(f\) is strictly bell-shaped. Conversely, any parameters $a, b, c$ and $\vphi$ with the properties listed above (where in $(c)$ we assume that $\mathcal{L}f(i\xi)$ is defined by \cref{eq:Fourier-exp-Kwa}) correspond to some weakly bell-shaped function $f$ (possibly with an extra atom at $b$).
\end{thm}

\begin{thm}[Corollary 1.6 in \cite{KS22}]\label{thm:Kwa-bound}
If $f$ is a strictly bell-shaped function, then there are constants $p$ and $q$ such that for every $n = 1,2,\dots$, all points at which $f^{(n)}$ changes sign are contained in $[pn,qn]$.
\end{thm}

\subsection{Proof of \cref{thm:num-zero}}

First, we give the expression of $G_{\{2\}^n}$ using the $\eta$-function defined by
$$\eta(\tau) \ceq q^{1/24}\prod_{m=1}^{\infty}(1-q^m)
\ ,\quad q \ceq e^{2\pi i\tau}.$$
Let $D$ denote $q\frac{d}{dq} = \frac{1}{2\pi i}\frac{d}{d\tau}$.

\begin{lem}\label{lem:eta-G}
For each integer $n\geq 0$, we have
\begin{equation}
G_{\{2\}^n}(\tau)
= \frac{8^n\pi^{2n}}{(2n+1)!}\frac{D^n(\eta(\tau)^3)}{\eta(\tau)^3}.
\end{equation}
\end{lem}

\begin{proof}
The equation is essentially the same as Equation (3.3) in \cite{AOS24} (see also \cite{AOS25}):
$$\frac{2^n}{(2n+1)!}\frac{D^n(\eta(\tau)^3)}{\eta(\tau)^3}
= \frac{1}{(2\pi i)^{2n}}\sum_{(1^{m_1},\dots,n^{m_n})\vdash n}\prod_{j=1}^{n}\frac{1}{m_j!}\left(-\frac{1}{j}G_{2j}(\tau)\right)^{m_j}.$$
Here, $\vdash$ indicates that $(1^{m_1},\dots,n^{m_n})$ is a partition of $n$, and we have made a slight modification to the formula to conform to the notation used in this paper. By \cref{eq:def-G}, we have
\begin{align*}
\sum_{n=0}^{\infty} G_{\{2\}^n}(\tau)X^{n}
&= \exp\left(\sum_{m=1}^{\infty}(-1)^{m-1}\frac{G_{2m}(\tau)}{m}X^{m}\right)
\\
&= \sum_{k=0}^{\infty}\left(\sum_{(1^{m_1},\dots,k^{m_k})\vdash k}\prod_{j=1}^{k}\frac{1}{m_j!}\left(-\frac{1}{j}G_{2j}(\tau)\right)^{m_j}(-X)^k\right),
\end{align*}
where the second equality follows from the generating function for P\'{o}lya's cycle index polynomials; see, for example, Lemma 2.1 in \cite{AGOS25}. Thus, we have
\begin{align*}
G_{\{2\}^n}(\tau)
&= (-1)^n\sum_{(1^{m_1},\dots,n^{m_n})\vdash n}\prod_{j=1}^{n}\frac{1}{m_j!}\left(-\frac{1}{j}G_{2j}(\tau)\right)^{m_j}
= \frac{8^n\pi^{2n}}{(2n+1)!}\frac{D^n(\eta(\tau)^3)}{\eta(\tau)^3}.
\end{align*}
This completes the proof.
\end{proof}

By \cref{lem:eta-G} and the fact that the $\eta$-function has neither zeros nor poles on $\HH$, the zeros of $G_{\{2\}^n}$ coincide with those of $D^n(\eta^3)$. By Jacobi's triple product, we have
\begin{align}\label{eq:eta-triple}
\eta(\tau)^3 = \sum_{m=0}^{\infty}(-1)^m(2m+1)q^{(2m+1)^2/8}.
\end{align}
Define the functions $F_0,F_{1/2}:\RR\to\RR$ by
$$F_{0}(y) \ceq \left\{\begin{array}{lc}
\displaystyle{\frac{4}{\pi}\eta\left(\frac{4iy}{\pi}\right)^3
= \frac{4}{\pi}\sum_{m=0}^{\infty}(-1)^m(2m+1)e^{-y(2m+1)^2}}
& (y>0)
\\
0 & (y\leq 0)
\end{array}\right.$$
and
$$F_{1/2}(y) \ceq \left\{\begin{array}{lc}
\displaystyle{\frac{2\sqrt{2}}{\pi e^{\pi i/8}}\eta\left(\frac{1}{2}+\frac{4iy}{\pi}\right)^3
= \frac{2\sqrt{2}}{\pi}\sum_{m=0}^{\infty}(-1)^{[m/2]}(2m+1)e^{-y(2m+1)^2}}
& (y>0)
\\
0 & (y\leq 0)
\end{array}\right.$$
Since $\eta^3$ is a cusp form of weight $3/2$, its exponential decay at the cusps implies that, for every cusp $\alpha$,
$$\lim_{\tau\to \alpha}D^n(\eta^3)(\tau) = 0.$$
Hence, the functions $F_{0}$ and $F_{1/2}$ are smooth functions on $\RR$. In what follows, we show that $F_{0}(y)$ and $F_{1/2}(y)$ are strictly bell-shaped, which implies that $G_{\{2\}^n}$ has exactly $n$ zeros on the imaginary axis and on the vertical half-line $\Real(\tau)=1/2$, respectively.

\begin{lem}\label{lem:0-bell-shaped}
Define
$$\vphi_0(s) = \sum_{m=0}^{\infty}\mathbf{1}_{[(2m+1)^2,\infty)}(s).$$
Then, $F_0$ and $\vphi_0$ satisfy the conditions of \cref{thm:Kwa-main}, and we have
$$\mathcal{L}F_0(i\xi) = \exp\left( -i\frac{\pi^2}{8}\xi
+ \int_{-\infty}^{\infty}K(i\xi,s) \varphi_0(s)\,ds \right).$$
Consequently, $F_{0}$ is weakly bell-shaped. Since $F_{0}$ is smooth, it is strictly bell-shaped.
\end{lem}

\begin{proof}
We only verify that $F_0$ and $\vphi_0$ satisfy the representation \cref{eq:Fourier-exp-Kwa} and condition (c) in \cref{thm:Kwa-main}. The remaining conditions are straightforward to check. For $\xi\in\RR$, by Fubini's theorem, we have
\begin{align*}
\mathcal{L}F_0(i\xi)
= \int_{-\infty}^{\infty}e^{-i\xi y}F_0(y)\, dy
&= \frac{4}{\pi}\lim_{\ep\to +0}\sum_{m=0}^{\infty}(-1)^m(2m+1)\int_{\ep}^{\infty}e^{-(i\xi+(2m+1)^2)y}\, dy
\\
&= \frac{4}{\pi}\lim_{\ep\to +0}e^{-i\xi\ep}\sum_{m=0}^{\infty}(-1)^m\frac{(2m+1)}{i\xi+(2m+1)^2}e^{-(2m+1)^2\ep}.
\end{align*}
We now justify the interchange of the limit and the infinite sum in the last line.
Using the identity
\begin{align*}
\frac{2m+1}{i\xi+(2m+1)^2}
&= \frac{1}{2m+1} - \frac{i\xi}{(2m+1)(i\xi+(2m+1)^2)},
\end{align*}
it suffices to consider separately
\begin{align*}
\lim_{\ep\to +0}\sum_{m=0}^{\infty}(-1)^m\frac{1}{2m+1}e^{-(2m+1)^2\ep}
\end{align*}
and
\begin{align*}
\lim_{\ep\to +0}\sum_{m=0}^{\infty}
(-1)^m\frac{i\xi}{(2m+1)(i\xi+(2m+1)^2)}e^{-(2m+1)^2\ep}.
\end{align*}
For the second expression, the limit and the infinite sum may be interchanged by the Weierstrass $M$-test, since, for $\ep\geq 0$,
\begin{align*}
\left|\frac{i\xi}{(2m+1)(i\xi+(2m+1)^2)}e^{-(2m+1)^2\ep}\right|
\leq \frac{|\xi|}{(2m+1)^3},
\end{align*}
and
$$\sum_{m=0}^{\infty}\frac{|\xi|}{(2m+1)^3}<\infty.$$
For the first expression, define
$$b_m(\ep) \ceq \frac{1}{2m+1}e^{-(2m+1)^2\ep}.$$
For each $\ep\geq 0$, the sequence $\{b_m(\ep)\}_{m\geq 0}$ is monotonically decreasing in $m$. Hence, by the alternating series estimate,
$$\left|\sum_{m=M}^{\infty}(-1)^mb_m(\ep)\right| \leq b_M(\ep) \leq \frac{1}{2M+1}\to 0 \quad \text{as}\ M\to\infty.$$
Therefore, the first series converges uniformly in $\ep\geq 0$, and hence the limit and the infinite sum may be interchanged, which gives
\begin{align*}
\mathcal{L}F_0(i\xi)
&= \frac{4}{\pi}\sum_{m=0}^{\infty}(-1)^m\frac{(2m+1)}{i\xi+(2m+1)^2}.
\end{align*}
By the partial fraction expansion of $1/\cosh$, we obtain
\begin{align*}
\mathcal{L}F_0(i\xi)
&= \frac{4}{\pi}\sum_{m=0}^{\infty}(-1)^m\frac{(2m+1)}{i\xi+(2m+1)^2}
= \frac{1}{\cosh(\frac{\pi}{2}\sqrt{i\xi})}.
\end{align*}
On the other hand, we have
\begin{align*}
\int_{-\infty}^{\infty}K(i\xi,s)\vphi_0(s)ds
&= \sum_{m=0}^{\infty}\int_{(2m+1)^2}^{\infty}\left(\frac{1}{i\xi+s}-\frac{1}{s}+\frac{i\xi}{s^2}\right) ds
\\
&= \sum_{m=0}^{\infty}\left[\log\left(1+\frac{i\xi}{s}\right)-\frac{i\xi}{s}\right]_{(2m+1)^2}^{\infty}
\\
&= \sum_{m=0}^{\infty} \left(\frac{i\xi}{(2m+1)^2}-\log\left(1+\frac{i\xi}{(2m+1)^2}\right)\right)
\\
&= i\frac{\pi^2}{8}\xi - \sum_{m=0}^{\infty}\log\left(1+\frac{i\xi}{(2m+1)^2}\right).
\end{align*}
The first equality follows from Fubini's theorem. Indeed, the estimate $|K(i\xi,s)|\leq \xi^2/s^3$ shows that $K(i\xi,s)\vphi_0(s)$ is absolutely integrable. Thus, by the Weierstrass product for $\cosh$, we obtain
\begin{align*}
\exp\left(- i\frac{\pi^2}{8}\xi + \int_{-\infty}^{\infty}K(i\xi,s)\vphi_0(s)ds\right)
&= \prod_{m=0}^{\infty}\left(1+\frac{i\xi}{(2m+1)^2}\right)^{-1}
= \frac{1}{\cosh(\frac{\pi}{2}\sqrt{i\xi})} = \mathcal{L}F_0(i\xi).
\end{align*}
Therefore, we get the representation \cref{eq:Fourier-exp-Kwa}. Condition (c) in \cref{thm:Kwa-main} follows from $\cosh(\frac{\pi}{2}\sqrt{i\xi})\neq 0$ for all $\xi\in\RR$.
\end{proof}

\begin{lem}\label{lem:partial-cosh}
For each $s\in\CC$, we have
\begin{align*}
\frac{\cosh(s\pi)}{\cosh(2s\pi)}
&= \frac{2\sqrt{2}}{\pi}\sum_{m=0}^{\infty}(-1)^{[m/2]}\frac{(2m+1)}{(4s)^2+(2m+1)^2}.
\end{align*}
\end{lem}

\begin{proof}
Using the exponential representation of $\cosh$, we obtain
\begin{align*}
\frac{\cosh(s)}{\cosh(2s)}
= \frac{1}{2\sqrt{2}}\left(\frac{1}{\cosh(s+\pi i/4)} + \frac{1}{\cosh(s-\pi i/4)}\right).
\end{align*}
Thus, by the partial fraction expansion of $1/\cosh$, we have
\begin{align*}
\frac{\cosh(\pi s)}{\cosh(2\pi s)}
&= \frac{\sqrt{2}}{\pi}\sum_{m=0}^{\infty}(-1)^m(2m+1)\left(\frac{1}{(2s+i/2)^2+(2m+1)^2} + \frac{1}{(2s-i/2)^2+(2m+1)^2}\right).
\end{align*}
Using the identity
\begin{align*}
a\left(\frac{1}{(x+i/2)^2+a^2} + \frac{1}{(x-i/2)^2+a^2}\right)
&= 2\left(\frac{2a-1}{(2x)^2+(2a-1)^2} + \frac{2a+1}{(2x)^2+(2a+1)^2}\right),
\end{align*}
which can be verified by putting both sides over a common denominator, we obtain
\begin{align*}
\frac{\cosh(\pi s)}{\cosh(2\pi s)}
&= \frac{2\sqrt{2}}{\pi}\sum_{m=0}^{\infty}(-1)^m\left(\frac{4m+1}{(4s)^2+(4m+1)^2} + \frac{4m+3}{(4s)^2+(4m+3)^2}\right)
\\
&= \frac{2\sqrt{2}}{\pi}\sum_{m=0}^{\infty}(-1)^{[m/2]}\frac{2m+1}{(4s)^2+(2m+1)^2}.
\end{align*}
This completes the proof.
\end{proof}

\begin{lem}\label{lem:1/2-bell-shaped}
Define
$$\vphi_{1/2}(s) = \sum_{m=0}^{\infty}\mathbf{1}_{[(2m+1)^2,\infty)}(s)
- \sum_{m=0}^{\infty}\mathbf{1}_{[4(2m+1)^2,\infty)}(s).$$
Then, $F_{1/2}$ and $\vphi_{1/2}$ satisfy the conditions of \cref{thm:Kwa-main}, and we have
$$\mathcal{L}F_{1/2}(i\xi) = \exp\left( -i\frac{3\pi^2}{32}\xi
+ \int_{-\infty}^{\infty}K(i\xi,s) \varphi_{1/2}(s)\,ds \right).$$
Consequently, $F_{1/2}$ is weakly bell-shaped. Since $F_{1/2}$ is smooth, it is strictly bell-shaped.
\end{lem}

\begin{proof}
Since the proof is similar to that of \cref{lem:0-bell-shaped}, we omit the details and only sketch the argument. The interchange of the limits is justified, by \cref{lem:partial-cosh}, we obtain
\begin{align*}
\mathcal{L}F_{1/2}(i\xi)
= \int_{-\infty}^{\infty}e^{-i\xi y}F_{1/2}(y)\, dy
&= \frac{2\sqrt{2}}{\pi}\sum_{m=0}^{\infty}\int_{0}^{\infty}(-1)^{[m/2]}(2m+1)e^{-(i\xi+(2m+1)^2)y}\, dy
\\
&= \frac{2\sqrt{2}}{\pi}\sum_{m=0}^{\infty}(-1)^{[\frac{m}{2}]}\frac{2m+1}{i\xi+(2m+1)^2}
\\
&= \frac{\cosh(\sqrt{i\xi}\pi/4)}{\cosh(\sqrt{i\xi}\pi/2)}.
\end{align*}
On the other hand, we have
\begin{align*}
\int_{-\infty}^{\infty}K(i\xi,s)\vphi_{1/2}(s)ds
&= i\frac{3\pi^2}{32}\xi + \sum_{m=0}^{\infty}\log\left(\frac{1+\frac{i\xi}{4(2m+1)^2}}{1+\frac{i\xi}{(2m+1)^2}}\right).
\end{align*}
Thus, by the Weierstrass product for $\cosh$, we have
\begin{align*}
\exp\left(- i\frac{3\pi^2}{32}\xi + \int_{-\infty}^{\infty}K(i\xi,s)\vphi_{1/2}(s)ds\right)
&= \prod_{m=0}^{\infty}\left(\frac{1+\frac{i\xi}{4(2m+1)^2}}{1+\frac{i\xi}{(2m+1)^2}}\right)
= \frac{\cosh(\sqrt{i\xi}\pi/4)}{\cosh(\sqrt{i\xi}\pi/2)}
= \mathcal{L}F_{1/2}(i\xi).
\end{align*}
Therefore, we get the representation \cref{eq:Fourier-exp-Kwa}.
\end{proof}

\begin{restatedthm}{thm:num-zero}
For each integer $n>0$, the function $G_{\{2\}^n}$ has exactly $n$ simple zeros lying on $\{iy\mid y\in \RR_{>0}\}$ and $\{1/2+iy\mid y\in \RR_{>0}\}$, respectively.
\end{restatedthm}

\begin{proof}
Here, we prove only that $G_{\{2\}^n}$ has exactly $n$ simple zeros on the imaginary axis. The proof for the vertical half-line $\{\tau\in\HH \mid \Real(\tau)=1/2\}$ is entirely analogous and is therefore omitted. By the above discussion, it suffices to show that $F_0^{(n)}$ has exactly $n$ simple zeros on $(0,\infty)$.

Since $F_0$ is strictly bell-shaped, $F_0^{(n)}$ changes sign exactly $n$ times on $(0,\infty)$. Hence, $F_0^{(n)}$ has at least $n$ distinct zeros on $(0,\infty)$. Suppose that $F_0^{(n)}$ has more than $n$ distinct zeros or has a multiple zero. Using the fact that $F_0^{(n)}$ tends to $0$ as $x\to0$ and as $x\to\infty$, together with Rolle's theorem, we then find that $F_0^{(n+1)}$ has at least $n+2$ distinct zeros.

We now use the fact that, on $(0,\infty)$, $F_0$ agrees with the restriction of a holomorphic function. In particular, $F_0^{(n+1)}$ cannot vanish identically on any nonempty open subinterval of $(0,\infty)$. Let
\[
0<x_1<\cdots<x_{n+2}
\]
be distinct zeros of $F_0^{(n+1)}$. Since $F_0^{(n+1)}$ tends to $0$ as $x\to0$ and as $x\to\infty$, each of the intervals
\[
(0,x_1),\quad (x_1,x_2),\quad \ldots,\quad
(x_{n+1},x_{n+2}),\quad (x_{n+2},\infty)
\]
contains a nonzero local extremum of $F_0^{(n+1)}$. At each such point, $F_0^{(n+2)}$ changes sign. Thus, $F_0^{(n+2)}$ has at least $n+3$ zeros at which it changes sign. This contradicts the fact that $F_0$ is strictly bell-shaped. Therefore, $F_0^{(n)}$ has exactly $n$ zeros on $(0,\infty)$, all of which are simple.
\end{proof}

It follows from \cref{thm:Kwa-bound,lem:0-bell-shaped,lem:1/2-bell-shaped} that there exist constants $M_0,M_{1/2}>0$ such that, for every $n=1,2,\dots$, all zeros of $F_{0}^{(n)}$ and $F_{1/2}^{(n)}$ in $(0,\infty)$ lie in $(0,nM_0]$ and $(0,nM_{1/2}]$, respectively. Hence, for each fixed $r$, the sequence ${y_r^{(*,n)}/n}_{n\geq r}$ is bounded and thus admits a convergent subsequence. In the next section, we show that the full sequence in fact converges to a specific value.

Finally, we apply \cref{thm:num-zero} to investigate the number of sign changes in the Fourier coefficients of $G_{\{2\}^n}$. Although this analysis does not provide a definitive resolution of the conjecture stated below, it yields a partial result supporting the conjecture.

We say that a sequence $(a_n)_{n\in\ZZ}$ changes sign exactly \(n\) times if \(n+1\) is the maximal length \(m\) of a strictly increasing sequence $n_1<n_2<\cdots<n_m$ in \(\ZZ\) such that $a_{n_i}a_{n_{i+1}}<0$ for \(j=1,\ldots,m-1\).

\begin{conj}\label{conj:sign-change-G}
Write the Fourier expansion of $G_{\{2\}^n}$ as
$$G_{\{2\}^n}(\tau) = \sum_{m=0}^{\infty}a_{m}q^{m}.$$
Then, $a_m\neq 0$ for all $m\geq 0$, and the sequence $(a_m)_{m\geq 0}$ changes the sign exactly $n$ times.
\end{conj}

\begin{lem}\label{lem:sign-change}
If the sequence $(a_m)_{m\geq 0}$ has exactly $n$ sign changes, then the function
\[
f(x)=\sum_{m=0}^{\infty} a_m x^m
\]
has at most $n$ positive zeros.
\end{lem}

\begin{proof}
We proceed by induction on $n$. The case $n=0$ is trivial. Assume that the statement holds for $n=\mu-1$, and suppose that $(a_m)_{m\geq 0}$ has $\mu$ sign changes. Without loss of generality, we may assume that $a_0\neq 0$. Let $M$ be the index immediately preceding the first sign change. Define
$$g(x) = x^{M+1}\frac{d}{dx}\left(x^{-M}f(x)\right)
= x^{M+1}\frac{d}{dx}\left( x^{-M}\sum_{m=0}^{\infty}a_mx^m \right)
= \sum_{m=0}^{\infty}b_mx^m,$$
where $b_m=(m-M)a_m$. Then the sequence $(b_m)_{m\geq 0}$ has at most $\mu-1$ sign changes. By the induction hypothesis, $g$ has at most $\mu-1$ positive zeros. Since $x^{M+1}>0$ for $x>0$, it follows that
$$\frac{d}{dx}\left(x^{-M}f(x)\right)$$
also has at most $\mu-1$ positive zeros. Hence, by Rolle's theorem, $x^{-M}f(x)$ has at most $\mu$ positive zeros. Since $x^{-M}\neq 0$ for $x>0$, the same is true of $f(x)$. This completes the induction and hence the proof.
\end{proof}

\begin{cor}
Let $(a_m)_{m\geq 0}$ be the sequence given in \cref{conj:sign-change-G}. Then, the sequence changes sign at least $n$ times.
\end{cor}

\begin{proof}
This follows immediately from \cref{thm:num-zero} and \cref{lem:sign-change}.
\end{proof}

\section{Asymptotic of zeros}\label{sec:asymptotic-zero}

In this section, we give the proof of \cref{thm:lim-zero,thm:lim-zero-gen}.

\begin{lem}\label{lem:sgn-series}
Let $x>1$, and let $(a_m)_{m\geq 0}$ and $(b_m)_{m\geq 0}$ be sequences of real numbers. For $n\geq 1$, define
$$A_n = \sum_{m=0}^{\infty} a_m x^{n b_m},$$
and assume that the series converges absolutely when $n=1$. Suppose that there exists an index $k$ with $a_k\neq 0$ such that $b_k>b_m$ for every $m\neq k$ with $a_m\neq 0$. Then, for all sufficiently large $n$, $A_n \neq 0$ and $\sgn A_n=\sgn a_k$.
\end{lem}

\begin{proof}
Dividing $A_n$ by the positive quantity $x^{n b_k}$, we obtain
$$x^{-n b_k}A_n = a_k+\sum_{\substack{m\geq0\\m\neq k}} a_m x^{-n(b_k-b_m)}.$$
For every $m\neq k$ with $a_m\neq0$, we have $b_k-b_m>0$, and hence
$$a_m x^{-n(b_k-b_m)}\longrightarrow 0 \qquad (n\to\infty).$$
Moreover, for $n\geq1$,
$$\left|a_m x^{-n(b_k-b_m)}\right| \leq |a_m|x^{-(b_k-b_m)} = x^{-b_k}|a_m|x^{b_m}.$$
Since
$$\sum_{m=0}^{\infty}|a_m|x^{b_m}<\infty,$$
the dominated convergence theorem yields
$$\lim_{n\to\infty}x^{-n b_k}A_n=a_k.$$
Since $x^{n b_k}>0$, it follows that, for all sufficiently large $n$, $A_n \neq 0$ and $\sgn A_n=\sgn a_k$.
\end{proof}

\begin{restatedthm}{thm:lim-zero}[Restated on the imaginary axis]
Write the zeros of $G_{\{2\}^n}(\tau)$ on the imaginary axis as $iy_1^{(n)},iy_2^{(n)},\dots,iy_n^{(n)}$, where $y_1^{(n)}>y_2^{(n)}>\cdots>y_n^{(n)}>0$. Then, for any integer $r>0$, we have
$$\lim_{n\to \infty} \left( y_{r}^{(n)} - \left(n+\frac{1}{2}\right)\frac{1}{\pi r}\log\left(\frac{2r+1}{2r-1}\right)\right) = 0.$$
\end{restatedthm}

\begin{proof}
Let $\til{y}_r^{(n)}\ceq y_r^{(n)}/n$. By \cref{lem:eta-G} and \cref{eq:eta-triple}, $\til{y}_r^{(n)}$ is the $r$-th largest zero of the function
\begin{align*}
f_n:\RR_{>0}\to \RR\ ;\ f_n(y) \ceq \left(\frac{\pi}{4}\right)^n\sum_{m=0}^{\infty}(-1)^m(2m+1)^{2n+1}e^{-\pi ny(2m+1)^2/4}.
\end{align*}
Thus it suffices to show that
$$\til{y}_{r}^{(n)} = \left(1+\frac{1}{2n}\right)\frac{1}{\pi r}\log\left(\frac{2r+1}{2r-1}\right) + o\left(\frac{1}{n}\right).$$
Write $\lambda_m=\frac{\pi}{4}(2m+1)^2$ and $\Phi_m(y)=\log\lambda_m-\lambda_m y$. Then we have
\begin{align*}
f_n(y) = \sum_{m=0}^{\infty}(-1)^m(2m+1)\lambda_m^{n} e^{- ny\lambda_m}
= \sum_{m=0}^{\infty}(-1)^m(2m+1) e^{n\Phi_m(y)}.
\end{align*}
For each integer $r>0$, let $x_r\in\RR_{>0}$ be the unique value satisfying
$$\Phi_{r-1}(x_r) = \Phi_r(x_r).$$
Equivalently, $\log(\lambda_r/\lambda_{r-1}) = (\lambda_r-\lambda_{r-1})x_r$. Simplifying, we obtain
$$x_r = \frac{1}{\pi r}\log\left(\frac{2r+1}{2r-1}\right).$$
Note that $x_1>x_2>\dots>0$. On the other hand, we have
\begin{equation}\label{eq:ineq-Phi}
\begin{split}
\Phi_r(x)-\Phi_{r-1}(x) &= \Phi_r(x)-\Phi_{r-1}(x) - \Phi_r(x_r)+\Phi_{r-1}(x_r)
\\
&= -\lambda_r x + \lambda_{r-1} x + \lambda_r x_r - \lambda_{r-1} x_r
= (\lambda_r - \lambda_{r-1})(x_r-x)
= 2\pi r(x_r-x).
\end{split}
\end{equation}
Thus, for $x_{r}<y<x_{r-1}$, we have
$$\Phi_{0}(y)<\cdots<\Phi_{r-2}(y)<\Phi_{r-1}(y)>\Phi_r(y)>\Phi_{r+1}(y)>\cdots$$
and
$$\Phi_{0}(x_r)<\cdots<\Phi_{r-2}(x_r)<\Phi_{r-1}(x_r)=\Phi_r(x_r)>\Phi_{r+1}(x_r)>\cdots .$$
Let $M_r=\Phi_{r-1}(x_r)=\Phi_r(x_r)$. Define
$$F_{n,r}(t) \ceq e^{-nM_r}f_n\left(x_r+\frac{t}{n}\right)
= \sum_{m=0}^{\infty}(-1)^m(2m+1) e^{n(\Phi_m(x_r)-M_r)}e^{-\lambda_m t}.$$
Since $\Phi_m(x_r)-M_r=0$ for $m=r-1,r$, while $\Phi_m(x_r)-M_r<0$ for $m\neq r-1,r$, the function $F_{n,r}(t)$ converges uniformly on every compact subset of $\RR_{>0}$ to 
$$\til{F}_r(t) = (-1)^{r}((2r+1) e^{-\lambda_r t}-(2r-1) e^{-\lambda_{r-1} t})$$
as $n\to\infty$. By a straightforward calculation, the unique zero of $\til{F}_r(t)$ is given by
$$t_r = \frac{1}{2\pi r}\log\left(\frac{2r+1}{2r-1}\right) = \frac{x_r}{2},$$
and this zero is simple. Therefore, for sufficiently small $\ep>0$ and all sufficiently large $n$, $F_{n,r}(t)$ has a unique simple zero in $(t_r-\ep,t_r+\ep)$. Denote this zero by $t_{n,r}$. Then $x_{n,r}=x_r+t_{n,r}/n$ is a zero of $f_n(y)$. Moreover, from
$$\lim_{n\to\infty} n(x_{n,r}-x_r)
= \lim_{n\to\infty} t_{n,r} = t_r$$
and $t_r=x_r/2$, we obtain
$$x_{n,r} = x_r + \frac{x_r}{2n} + o\left(\frac{1}{n}\right).$$
Finally, we prove that $\til{y}_r^{(n)}=x_{n,r}$ for sufficiently large $n$. It is sufficient to show that $x_{n,r}$ is the $r$-th largest zero. By choosing $\ep>0$ sufficiently small, we may assume that the neighborhoods
$$U_i = (x_i-\ep, x_i+\ep)\ ,\quad i=1,\dots,r$$
are pairwise disjoint. By the above argument, for all sufficiently large $n$, $f_n(y)$ has a unique simple zero $x_{n,j}$ in each $U_j$. On each connected component
$$I\subset (x_r-\ep,\infty)\setminus\bigcup_{i=1}^{r}U_i,$$
there exists an integer $m$ with $1\leq m\leq r$ such that
$$\Phi_{0}(y)<\cdots<\Phi_{m-2}(y)<\Phi_{m-1}(y)>\Phi_m(y)>\Phi_{m+1}(y)>\cdots $$
for all $y\in I$. Therefore, using \cref{lem:sgn-series}, $f_n$ does not vanish on $I$ for sufficiently large $n$. Consequently, we obtain
$$\til{y}_{r}^{(n)} = \left(1+\frac{1}{2n}\right)x_r + o\left(\frac{1}{n}\right).$$
This completes the proof.
\end{proof}

\begin{restatedthm}{thm:lim-zero}[Restated on the half-line $1/2+i\RR_{>0}$]
Write the zeros of $G_{\{2\}^n}(\tau)$ on the half-line $\{\tau\in\HH\mid\Real(\tau)=1/2\}$ as $\frac{1}{2}+iy_1^{(n)},\frac{1}{2}+iy_2^{(n)},\dots,\frac{1}{2}+iy_n^{(n)}$, where $y_1^{(n)}>y_2^{(n)}>\cdots>y_n^{(n)}>0$. Then, for any integer $r>0$, we have
$$\lim_{n\to \infty} \left( y_{r}^{(n)} - \left(n+\frac{1}{2}\right)\frac{1}{2\pi r}\log\left(\frac{4r+1}{4r-1}\right)\right) = 0.$$
\end{restatedthm}

\begin{proof}
The proof is exactly the same as above, with $f_n$ replaced by
\begin{align*}
g_n:\RR_{>0}\to \RR\ ;\ g_n(y) \ceq \left(\frac{\pi}{4}\right)^n\sum_{m=0}^{\infty}(-1)^{[m/2]}(2m+1)^{2n+1}e^{-\pi ny(2m+1)^2/4}.
\end{align*}
Note that the only difference is the sign pattern of the terms in the infinite sum.
\end{proof}

\begin{restatedthm}{thm:lim-zero-gen}
For any integers $r>0$ and $j$, there exists a sequence $\{z_{r,j}^{(n)}\}_{n\geq r}$ of zeros of $G_{\{2\}^n}$, such that
$$\lim_{n\to \infty} \left( z_{r,j}^{(n)} - \left(\frac{j}{r} + i\left(n+\frac{1}{2}\right)\frac{1}{\pi r}\log\left(\frac{2r+1}{2r-1}\right)\right)\right) = 0.$$
\end{restatedthm}

\begin{proof}
If we carry out the argument in the first half of the proof of \cref{thm:lim-zero} without restricting to the imaginary axis, then the function $f_n$ is replaced by
\begin{align*}
h_n:\HH\to \CC\ ;\ h_n(\tau) = \left(\frac{\pi}{4}\right)^n\sum_{m=0}^{\infty}(-1)^m(2m+1)^{2n+1}e^{n\Phi_m(y) + i\lambda_m x}, \quad \tau=x+iy
\end{align*}
Accordingly, for each $r>0$, setting
$$H_{n,r}(x,t) = e^{-nM_r}h_n\left(x+i\left(x_r+\frac{t}{n}\right)\right),$$
we find that $H_{n,r}(x,t)$ converges to
$$\til{H}_r(x,t) = (-1)^{r}((2r+1) e^{-\lambda_r t + i\lambda_rx}-(2r-1) e^{-\lambda_{r-1} t + i\lambda_{r-1}x}).$$
The zeros of $\widetilde{H}_r$ are given by
$$x = \frac{j}{r}
\quad \text{and}\quad
t=\frac{1}{2\pi r}\log\left(\frac{2r+1}{2r-1}\right).$$
The same argument as in the proof of \cref{thm:lim-zero} then yields the desired sequence of zeros of $G_{\{2\}^n}(\tau)$. This completes the proof.
\end{proof}

\section{Global properties of zeros}\label{sec:global-zero}

In this section, we give the proof of \cref{thm:equiv-zero,thm:simple-zero,thm:trans-zero}. Let $E_k(\tau)\ceq \frac{1}{\zeta(k)}G_k(\tau)$ for even $k\geq 2$. Then all the Fourier coefficients of $E_k$ are rational. In particular, its constant term is $1$. In the following arguments, we apply statements formulated in terms of $E_k$ to the corresponding statements for $G_k$. Indeed, since
$$
\zeta(k)=-\frac{1}{2}\frac{(2\pi i)^kB_k}{k!},
\quad B_k:k\text{-th Bernoulli number}
$$
for even $k>0$, the quantity $G_k/(2\pi i)^k$ is a rational multiple of $E_k$. Therefore, such applications are justified after multiplying by a suitable power of $2\pi i$.

\subsection{Simplicity of zeros}

\begin{thm}[Corollary 2 of Gun--Oesterl\'{e}\cite{GO22}]\label{lem:GO-poly}
Let $P$ be a nonzero element of the unique factorization domain $\overline{\QQ}[X,Y,Z]$. Let $a$ be a
zero of $\psi_P=P(E_2,E_4,E_6)$ in $\HH$ and $e$ be its order. There exists an irreducible factor $R$ of $P$, unique up to multiplication by a scalar, such that $\psi_R(a)=0$, where $\psi_R=R(E_2,E_4,E_6)$. The function $\psi_R$ has a simple zero at $a$ and the $R$-adic valuation of $P$, denoted by $v_R(P)$, is $e$.
\end{thm}

\begin{restatedthm}{thm:simple-zero}
For each integer $n>0$, all zeros of $G_{\{2\}^n}$ are simple.
\end{restatedthm}

\begin{proof}
Let $P_n(X,Y,Z)\in\QQ[X,Y,Z]$ be such that $G_{\{2\}^n}=P_n(G_2,G_4,G_6)$. Suppose that $a\in\HH$ is a zero of $G_{\{2\}^n}$ of multiplicity $e\geq 2$. Then, by \cref{lem:GO-poly}, there exists an irreducible factor $R$ of $P_n$ such that
$$R(G_2,G_4,G_6)(a) = 0
\quad\text{and}\quad
v_R(P_n) = e\geq 2.$$
Let $S=P_n/R^e$. First, we consider the case where
$$\frac{\partial R}{\partial X} \neq 0.$$
Then, we have
\begin{align}\label{eq:deri-P}
\frac{\partial P_n}{\partial X} = eR^{e-1}S\frac{\partial R}{\partial X} + R^e\frac{\partial S}{\partial X}.
\end{align}
Since $R$ is irreducible and $\frac{\partial R}{\partial X} \neq 0$, we obtain
$$v_R\left(\frac{\partial P_n}{\partial X}\right) =e-1.$$
On the other hand, from \cref{eq:der-G}, we have
\begin{equation}\label{eq:E_2-der}
\frac{\partial P_n}{\partial X} = P_{n-1}.
\end{equation}
Thus, $a$ is a zero of $G_{\{2\}^{n-1}}$ of multiplicity $e-1$. Hence, $a$ is a zero of $D(G_{\{2\}^{n-1}})$ of multiplicity $e-2$. Now, by \cref{lem:eta-G}, we obtain
\begin{align*}
D(G_{\{2\}^{n-1}}(\tau))
&= \frac{8^{n-1}\pi^{2n-2}}{(2n-1)!}\frac{D^n(\eta(\tau)^3)}{\eta(\tau)^3}
- \frac{8^{n-1}\pi^{2n-2}}{(2n-1)!}\frac{D^{n-1}(\eta(\tau)^3)}{\eta(\tau)^3}\frac{D(\eta(\tau)^3)}{\eta(\tau)^3}
\\
&= \frac{1}{4\pi^2}(n(2n+1)G_{\{2\}^{n}}(\tau)
- 3G_{2}(\tau)G_{\{2\}^{n-1}}(\tau)).
\end{align*}
Since $a$ is a zero of $G_{2}G_{\{2\}^{n-1}}$ of multiplicity at least $e-1$ and a zero of $D(G_{\{2\}^{n-1}})$ of multiplicity $e-2$, it follows that $a$ is a zero of $G_{\{2\}^{n}}$ of multiplicity $e-2$. This contradicts our initial assumption that $a$ is a zero of $G_{\{2\}^{n}}$ of multiplicity $e$. Second, we consider the case where 
$$\frac{\partial R}{\partial X} = 0.$$
Then, from \cref{eq:deri-P}, $R$ is an irreducible factor of
$$\frac{\partial^j P_n}{\partial X^j} \quad (j=0,1,2,\dots).$$
Therefore, from \cref{eq:E_2-der}, $R$ divides
$$\frac{\partial^n P_n}{\partial X^n} = P_0 = 1.$$
Hence $R$ is constant, which contradicts $R(G_2,G_4,G_6)(a)=0$.
\end{proof}

\subsection{$SL_2(\ZZ)$-inequivalence of zeros}

El Basraoui and Sebbar \cite{BS10} showed that $E_2$ has infinitely many zeros that are pairwise $SL_2(\ZZ)$-inequivalent, and Meher \cite{Meh13} proved the analogous result for quasimodular forms of depth one. A similar statement holds for quasimodular forms of maximal depth.

\begin{restatedthm}{thm:equiv-zero}
Let $f$ be a nonzero quasimodular form of weight $k$ and depth $k/2$. Suppose that $f$ has a simple zero. Then $f$ has infinitely many $SL_2(\ZZ)$-inequivalent zeros in the half-strip. In particular, the same holds for $G_{\{2\}^n}$ for every integer $n>0$.
\end{restatedthm}

\begin{proof}
We prove this by following the same strategy as in the proof of Theorem 3.5 in \cite{BS10}. From the quasimodular transformation, for $\tau\in\HH$ and $g=\begin{pmatrix}
a & b \\ c & d
\end{pmatrix}\in SL_2(\ZZ)$, we have
\begin{align}\label{trans-mod}
\frac{1}{(c\tau+d)^k}f(g\tau) = \sum_{r=0}^{k/2} \left(\frac{c}{c\tau+d}\right)^r f_r(\tau)\ ,\quad \text{where}\quad 0\neq f_{k/2}\in\CC.
\end{align}
Therefore, if the equation
$$\Psi(v,\tau) \ceq \sum_{r=0}^{k/2}(\tau+v)^{k/2-r}f_r(\tau) = 0$$
admits a solution $(v,\tau)\in\QQ\times\HH$, then $g\tau$ is a zero of $f(\tau)$, where $g=\begin{pmatrix}
a & b \\ c & d
\end{pmatrix}\in SL_2(\ZZ)$ with $c\neq 0$ is an element corresponding to $v=d/c$. We fix a simple zero $\tau_0$ of $f(\tau)$ and $g_0=\begin{pmatrix}
a_0 & b_0 \\ c_0 & d_0
\end{pmatrix}\in SL_2(\ZZ)$ with $c_0\neq 0$. Then, the pair $(d_0/c_0,g_0^{-1}\tau_0)$ is a solution of $\Psi=0$. In \cref{trans-mod}, multiplying by $(\tau+d/c)^{k/2}$ both sides, and then differentiating both sides with respect to $\tau$, we obtain
\begin{align*}
c^{-k/2}\frac{f'(g\tau)}{(c\tau+d)^{k/2+2}} - c^{1-k/2}\frac{k}{2}\frac{f(g\tau)}{(c\tau+d)^{k/2+1}}
= \frac{d}{d\tau}\sum_{r=0}^{k/2} \left(\tau+\frac{d}{c}\right)^{k/2-r} f_r(\tau).
\end{align*}
Thus, since $\tau_0$ is a simple zero of $f$, we have
$$\frac{d\Psi}{d\tau}(d_0/c_0,g^{-1}_0\tau_0) \neq 0.$$
Therefore, by the implicit function theorem, there exist a neighborhood $V\subset \RR$ of $d_0/c_0$ and a real-analytic function $\vphi:V\to\HH$ such that 
$$\vphi(d_0/c_0) = g_0^{-1}\tau_0
\quad\text{and}\quad
\Psi(v,\vphi(v)) = 0.$$
We claim that $\vphi$ is not constant on $V$. Indeed, if $\vphi$ were constant, then $\Psi(v,g_0^{-1}\tau_0)=0$ for all $v\in V$. It would follow that $\Psi(X-g_0^{–1}\tau_0,g_0^{-1}\tau_0)$ vanishes identically as a polynomial in $X$, which contradicts $f_{k/2}\neq 0$. Since $V\cap\QQ$ is infinite and $\vphi$ is not constant, the set $\{g_0\vphi(v)\mid v\in V\cap \QQ\}$ is infinite. Shrinking $V$ if necessary, we may assume that $g_0\vphi(V)$ is contained in a finite union of fundamental domains $\bigcup_{\gamma:\text{finite}}\gamma\mathcal{F}$. Thus, it follows that $f$ has infinitely many zeros $SL_2(\ZZ)$-inequivalent zeros. By applying a suitable $T$-transformation, these zeros can be moved into the half-strip.
\end{proof}

\subsection{Transcendence of zeros}\label{subsec:trans}

\begin{thm}[Chudnovsky\cite{Chu76}, generalized by Nesterenko\cite{Nes96}]\label{thm:CN}
For each $\tau\in\HH$, we have
$$\mathrm{trdeg}_{\QQ}\{E_2(\tau),E_4(\tau),E_6(\tau)\}\geq 2.$$
\end{thm}


\begin{restatedthm}{thm:trans-zero}
Let $f\in\overline{\QQ}[E_2,E_4,E_6]$ be a nonzero quasimodular form of weight $k$ and depth $k/2$. Then none of the zeros of $f$ in $\HH$ is a CM point. In particular, for every $n>0$, none of the zeros of $G_{\{2\}^n}$ is a CM point.
\end{restatedthm}

\begin{proof}
Let $\tau_0\in\HH$ be a zero of $f$.
Suppose, for contradiction, that
$\tau_0$ is a CM point. It is well known that $j(\tau_0)$ is algebraic; see, for example,
\cite[Proposition II.2.1]{Sil94}. By the classical identity
\begin{equation}\label{eq:def-j}
j(\tau) = 1728\frac{E_4(\tau)^3}{E_4(\tau)^3-E_6(\tau)^2},
\end{equation}
we have
$$\operatorname{trdeg}_{\QQ}\QQ\bigl(E_4(\tau_0),E_6(\tau_0)\bigr)\leq 1.$$
Since $f$ has weight $k$ and depth $k/2$, there exists a polynomial $P(X,Y,Z)\in\overline{\QQ}[X,Y,Z]$ such that
$$f(\tau) = P(E_2(\tau),E_4(\tau),E_6(\tau)),$$
and the coefficient of $X^{k/2}$ in $P$ is a nonzero element of $\overline{\QQ}$. In particular, $P(X,E_4(\tau_0),E_6(\tau_0))$ is a nonzero polynomial in $X$. Since $f(\tau_0)=0$, $E_2(\tau_0)$ is a root of $P(X,E_4(\tau_0),E_6(\tau_0))=0$. Therefore,
$$\operatorname{trdeg}_{\QQ}\QQ\bigl(E_2(\tau_0),E_4(\tau_0),E_6(\tau_0)\bigr)\leq 1,$$
which contradicts \cref{thm:CN}. Hence $\tau_0$ cannot be a CM point.
\end{proof}

\begin{remark}\label{Sebbarremark}
\cite[Corollary~1.10]{BP25} states that if $\tau$ is a zero of a quasimodular form with algebraic coefficients, then $j(\tau)$ is algebraic, and consequently $\tau$ is either a CM point or transcendental.
If this were true, then, together with \cref{thm:trans-zero}, it would imply that all zeros of $G_{{2}^n}$ are transcendental. However, the statement does not hold in the stated generality.
Indeed, let $\tau$ be a zero of $E_2$. Then \cite[Corollary~1.10]{BP25} would imply that $j(\tau)$ is algebraic. Since $E_2(\tau)=0$, it follows from \cref{eq:def-j} that
$$\operatorname{trdeg}_{\QQ}\QQ(E_2(\tau),E_4(\tau),E_6(\tau))\leq 1,$$
which contradicts \cref{thm:CN}. This remark is due to an observation by Professor Abdellah Sebbar.
\end{remark}

\begin{thm}[Gelfond\cite{Gel34}--Schneider\cite{Sch35}]\label{thm:GS}
If $a$ and $b$ are algebraic numbers with $a\neq 0,1$ and $b$ irrational, then the value $a^b$ is transcendental.
\end{thm}

\begin{cor}\label{cor:trans-lim-cor}
If $a$ is an algebraic real number, then the value $\frac{1}{\pi}\log a$ is transcendental. In particular, the limiting values for the zeros appearing in \cref{thm:lim-zero},
$$\frac{1}{\pi\kappa r}\log\left(\frac{2\kappa r+1}{2\kappa r-1}\right),
\qquad \kappa\in\{1,2\},$$
are transcendental.
\end{cor}

\begin{proof}
Suppose that $x=\frac{1}{\pi}\log a$ is algebraic. Then $-ix$ is an algebraic irrational number. Thus,
$$(-1)^{-ix} = e^{(-ix)\pi i} = e^{\log a} = a$$
is algebraic. This contradicts \cref{thm:GS}. Therefore, $x$ is transcendental.
\end{proof}

\section{Zeros inside Ford circle}\label{sec:fordcircle}

In this section, we prove \cref{thm:fordcircle}. Our argument follows a strategy similar to that used in \cite{IJT14}.

\begin{restatedthm}{thm:fordcircle}
Each Ford circle contains exactly two distinct (and simple) zeros of $G_{2,2}$.
\end{restatedthm}

\begin{lem}\label{lem:equiv-G22}
Let $\tau\in\HH$ and $g=\begin{pmatrix}
a & b \\ c & d
\end{pmatrix}\in SL_2(\ZZ)$ with $c\neq 0$. Then $G_{2,2}(g\tau)=0$ if and only if
$$\tau-\frac{\pi i}{G_2(\tau)+\sqrt{G_4(\tau)}} = -\frac{d}{c}
\quad\text{or}\quad
\tau-\frac{\pi i}{G_2(\tau)-\sqrt{G_4(\tau)}} = -\frac{d}{c}.$$
\end{lem}

\begin{proof}
From the transformation law of $G_2$:
\begin{align*}
(G_{2}|_{2}g)(\tau)
&= G_{2}(\tau) - \frac{c\pi i}{c\tau+d}
\end{align*}
and \cref{eq:def-G}, we obtain
\begin{align*}
(G_{\{2\}^n}|_{2n}g)(\tau)
&= \sum_{r=0}^{n}\frac{1}{(n-r)!}\left(-\frac{\pi i}{\tau+d/c}\right)^{n-r}G_{\{2\}^{r}}(\tau).
\end{align*}
This implies that
$G_{2,2}(g\tau)=0$ if and only if
\begin{align*}
0 &= G_{2,2}(\tau)
+ XG_{2}(\tau)
+ \frac{1}{2}X^2
\ ,\quad
X= -\frac{\pi i}{\tau+d/c}.
\end{align*}
Using the identity $(G_2)^2=2G_{2,2}+G_4$, it is equivalent to
$$X = -G_2(\tau)+\sqrt{G_4(\tau)}
\quad\text{or}\quad
X = -G_2(\tau)-\sqrt{G_4(\tau)}.$$
Equivalently,
$$\tau-\frac{\pi i}{G_2(\tau)+\sqrt{G_4(\tau)}} = -\frac{d}{c}
\quad\text{or}\quad
\tau-\frac{\pi i}{G_2(\tau)-\sqrt{G_4(\tau)}} = -\frac{d}{c}.$$
This completes the proof
\end{proof}

For simplicity, let $h_\pm(\tau):=G_2(\tau)\pm\sqrt{G_4(\tau)}$ and $g_\pm(\tau):=\tau-\frac{\pi i}{G_2(\tau)\pm\sqrt{G_4(\tau)}}=\tau-\frac{\pi i}{h_\pm(\tau)}$. 
Furthermore, we write $\tau=x+iy$ and let $g_\pm(\tau)=U_\pm(x,y)+iV_\pm(x,y)=U_\pm(\tau)+iV_\pm(\tau)$ and $C_\pm:=\zeta(2)\pm\sqrt{\zeta(4)}$. Moreover, we take branches of $\sqrt{E_4(\tau)}$ and $\sqrt{G_4(\tau)}$ for which the real parts are positive. We only consider the case $\Img\tau\geq 1$ below. In this region, the real part of $\sqrt{E_4(\tau)}$ is nonzero, since \cref{alpha} below gives $\alpha<1$. Hence, this choice of branch is well-defined.

\begin{lem}\label{lem:boundofEisenstein}
For each $\tau=x+iy\in\HH$, we have
\[
|E_2(\tau)-1|\le\frac{24r}{(1-r)^3}.
\]
Moreover, for each $\tau=x+iy\in\HH$ with $y \geq 1$, we have
\[
|\sqrt{E_4(\tau)}-1|\le1-\sqrt{1-\alpha},
\]
where $\alpha(y)\ceq\alpha\ceq\frac{240r(r^2+4r+1)}{(1-r)^5}$ and $r(y)\ceq r\ceq |q|$.
\end{lem}

\begin{proof}
The first inequality was obtained in \cite[Lemma 2.3]{IJT14}. Therefore, we only prove the second inequality. From the identity $E_4(\tau)=1+240\sum_{n\ge1}\frac{n^3q^n}{1-q^n}$, we see that
\[
|E_4(\tau)-1|\le240\sum\limits_{n\geq1}\frac{n^{3}r^{n}}{1-r^{n}}.
\]
On the other hand, applying $r\frac{d}{dr}$ to both sides of  $\sum_{n\geq0}r^n=\frac{1}{1-r}$ three times, we obtain
\[
\sum\limits_{n\geq0}n^3r^n=\frac{r(r^2+4r+1)}{(1-r)^4}.
\]
Therefore, 
\begin{equation}\label{alpha}
|E_4(\tau)-1|\le240\sum_{n\ge1}\frac{n^3r^n}{1-r^n}\le\frac{240}{1-r}\sum_{n\ge1}n^3r^n=\frac{240r(r^2+4r+1)}{(1-r)^5}= \alpha.
\end{equation}
Note that $0<\alpha<1$ for $y\ge1$. Indeed, $\alpha(y)$ is decreasing in $y$, and $\alpha(1)<1$.
We next estimate $|\sqrt{E_4(\tau)}-1|$. Using the identities
\begin{align*}
\Real(\sqrt{E_4(\tau)})^2-\Img(\sqrt{E_4(\tau)})^2
&=\Real(E_4(\tau)),\\
\Real(\sqrt{E_4(\tau)})^2+\Img(\sqrt{E_4(\tau)})^2
&=|E_4(\tau)|,
\end{align*}
together with $\Real(\sqrt{E_4(\tau)})>0$ and \cref{alpha}, we obtain
\[
\Real(\sqrt{E_4(\tau)})
= \sqrt{\frac{|E_4(\tau)|+\Real(E_4(\tau))}{2}}\ge
\sqrt{\Real(E_4(\tau))}\ge\sqrt{1-\alpha}.
\]
Thus, we obtain
\[ |\sqrt{E_4(\tau)}+1|\ge1+\Real(\sqrt{E_4(\tau)})\ge1+\sqrt{1-\alpha}. \]
As a consequence, we obtain
\begin{equation*}\label{sqrtE4bound}
|\sqrt{E_4(\tau)}-1|=\frac{|E_4(\tau)-1|}{|\sqrt{E_4(\tau)}+1|}\le \frac{\alpha}{1+\sqrt{1-\alpha}} = 1-\sqrt{1-\alpha}.
\end{equation*}
\end{proof}

By \cref{lem:boundofEisenstein}, we have
\begin{equation}\label{My}
|h_\pm(\tau)-C_\pm|\le\zeta(2)\frac{24r}{(1-r)^3}+\sqrt{\zeta(4)}\left(1-\sqrt{1-\alpha}\right) \eqcolon M(y).
\end{equation}
Note that both $\frac{24r}{(1-r)^3}$ and
$1-\sqrt{1-\alpha}$ are decreasing in $y$.
Hence, $M(y)$ is also decreasing in $y$.
A direct calculation at $y=1$ shows that $M(1)<C_-<C_+$. Therefore, $M(y)\le M(1)<C_-<C_+$ for all $y\ge1$.

\begin{lem}\label{lem1.04}
For each $\tau=x+iy\in \HH$ with $y\ge1.04$, we have
\[
|g_+'(\tau)-1|<1.
\]
\end{lem}

\begin{proof}
From $g_+(\tau)=\tau-\frac{\pi i}{h_+(\tau)}$, we have
\[
|g_+'(\tau)-1|=\frac{\pi |h_+'(\tau)|}{|h_+(\tau)^{2}|}.
\]
Thus, we need to compute a lower bound for $|h_+(\tau)|$ and an upper bound for $|h_+'(\tau)|$ for $\Img(\tau)\geq 1.04$. By \cref{My}, we obtain
\[
|h_+(\tau)|\ge C_+-M(y).
\]
Next, from $h_+(\tau)=G_2(\tau)+\sqrt{G_4(\tau)}$, we have 
\begin{align}\label{upperboundhplus}
|h_+'(\tau)|\le\zeta(2)|E_2'(\tau)|+\frac{\sqrt{\zeta(4)}}{2}\frac{|E_4'(\tau)|}{|\sqrt{E_4(\tau)}|}.
\end{align}
Using the same argument as \cref{alpha} yields
\[
|E_2'(\tau)|\le48\pi\frac{r(1+r)}{(1-r)^5} \;\text{ and }\; |E_4'(\tau)|\le480\pi\frac{r(1+11r+11r^2+r^3)}{(1-r)^7}.
\]
Moreover, by \cref{alpha}, we have $|\sqrt{E_4(\tau)}|\ge\sqrt{1-\alpha}$.
Combining the above estimates yields
\[
|h_+'(\tau)|\le\zeta(2)48\pi\frac{r(1+r)}{(1-r)^5}+240\pi\sqrt{\zeta(4)}\frac{r(1+11r+11r^2+r^3)}{\sqrt{1-\alpha}(1-r)^7}.
\]
Finally, we have
\[
|g_+'(\tau)-1|=\pi\frac{|h_+'(\tau)|}{|h_+(\tau)^2|}\le\pi\frac{\zeta(2)48\pi\frac{r(1+r)}{(1-r)^5}+240\pi\sqrt{\zeta(4)}\frac{r(1+11r+11r^2+r^3)}{\sqrt{1-\alpha}(1-r)^7}}{(C_+-M(y))^2}.
\]
Since the denominator of the fraction on the right-hand side is increasing in $y$, while the numerator is decreasing in $y$, the right-hand side is decreasing as a function of $y$. Moreover, substituting $y=1.04$ shows that the right-hand side is less than $1$. Therefore, it is less than $1$ for all $y\geq 1.04$.
\end{proof}

\begin{lem}\label{partialgplus}
For each $\tau=x+iy\in \HH$ with $y\ge1.04$, we have $\Real g_+'(\tau)=\partial_yV_+(x,y)>0$.
\end{lem}

\begin{proof}
Recall that, by the Cauchy–Riemann equations, $\Real g_+'(\tau)=\partial_y V_+(x,y)$. By \cref{lem1.04}, we have
\[
|\Real g_+'(\tau)-1|<1,
\] 
which implies $\Real g_+'(\tau)>0$.
\end{proof}

\begin{lem}\label{imagegplus}
For each $\tau=x+iy\in \HH$ with $y\in[1,1.04]$, we have
\[
V_+(x,y)<0.
\]
\end{lem}

\begin{proof}
By definition of $V_+$, we have
\[
V_+(x,y)=y-\Real\Big(\frac{\pi}{h_+(\tau)}\Big).
\]
We first claim 
\[
\Real\Big(\frac{1}{h_+}\Big)\ge\frac{1}{C_++M(y)}.
\]
To prove this, we let $h_+(\tau)=a+bi$ for some $a,b\in\RR$. Since $\Real(\frac{1}{h_+})=\frac{a}{a^2+b^2}$, our claim is equivalent to
\begin{align*}
a\big(C_++M(y)\big)\geq a^2+b^2.
\end{align*}
On the other hand, by \cref{My} we have
\[
(a-C_+)^2+b^2\le M(y)^{2}.
\]
Therefore, it suffices to show that
\[
M(y)^2-C_+^2+2aC_+\leq aC_++aM(y).
\]
Or equivalently,
\[
\big(C_+-M(y)\big)\big(C_++M(y)-a\big)\ge0,
\]
which follows from the facts that $M(y)<C_+$ and \cref{My}. Finally, we obtain an upper bound for $V_+(\tau)$
\[
V_+(\tau)\le y-\frac{\pi}{C_++M(y)}.
\]
Note that $M(y)$ is decreasing.
To obtain an estimate, we divide the interval $[1,1.04]$ into two subintervals $[1,1.02]$ and $[1.02,1.04]$. The corresponding uniform upper bounds are given by
\[
1.02-\frac{\pi}{C_++M(1)} \;\;\text{ and }\;\; 1.04-\frac{\pi}{C_++M(1.02)},
\]
respectively. Since both are negative, we conclude that $V_+(\tau)<0$ for all $\Img(\tau)\in[1,1.04]$.
\end{proof}

The corresponding argument for $g_-$ is entirely the same, so we omit the proof and record only the results.

\begin{lem}
The following are true.
\begin{itemize}
    \item For $\tau=x+iy\in\HH$ with $y\geq1.46$, we have $|g_-'(\tau)-1|<1.$
    \item For $\tau=x+iy\in\HH$ with $y\geq1.46$, we have $\Real g_-'(\tau)=\partial_y V_-(x,y)>0$.
    \item For $\tau=x+iy\in\HH$ with $y\in[1,1.46]$, we have $V_-(x,y)<0$.
\end{itemize}
\end{lem}

As before, we only treat the case of $g_+$, since the argument for $g_-$ is essentially the same.
By the definition of $V_+(x,y)$, we have
\[
\lim_{y\to\infty}V_+(x,y)=+\infty.
\]
Therefore, by \cref{imagegplus} and the intermediate value theorem, for each $x\in\RR$ there exists $y\in(1.04,\infty)$ for which $V_+(x+iy)=0$. Moreover, \cref{partialgplus} ensures that it is unique. Thus, we have a function $\phi_+:\RR\rightarrow (1.04,\infty)$ where $\phi_+(x)=y$ such that $V_+(x+iy)=0$. Thus, we have
\[
A_+:=\{\tau\in\HH\mid \Img\tau \geq 1,~|\Real\tau|\leq 1/2,~V_+(\tau)=0\}=\{x+i\phi_+(x)\mid x\in[-1/2,1/2]\}.
\]
Furthermore, by \cref{partialgplus} and the implicit function theorem, $y=\phi_+(x)$ is differentiable on $\RR$, and $\phi_+'(x)=-\frac{\partial_xV_+}{\partial_yV_+}$. Then, by the Cauchy-Riemann equation, we see that
\begin{align*}
\frac{d}{dx}(U_+(x,\phi_+(x)))&=(\partial_xU_+)(x,\phi_+(x))+\phi_+'(x)(\partial_yU_+)(x,\phi_+(x))
\\
&=\partial_xU_+-\partial_yU_+\frac{\partial_xV_+}{\partial_yV_+}
\\
&=\frac{\partial_xU_+\partial_yV_+-\partial_yU_+\partial_xV_+}{\partial_yV_+}
\\
&=\frac{(\partial_yV_+)^2+(\partial_xV_+)^2}{\partial_yV_+}>0
\end{align*}
for any $x\in\RR$. Hence, $U_+$ restricted to $A_+$ is increasing as a function of $x$ on $[-1/2,1/2]$. Similarly, we can define $A_-$ and $\phi_-$.

\begin{lem}\label{lem:gplusAplus}
We have $g_+(A_+)=[-1/2,1/2]$ and $g_-(A_-)=[-1/2,1/2]$.
\end{lem}

\begin{proof}
We only treat the case of $g_+$. The case of $g_-$ follows in the same way.
By the definition of $A_+$, we have $g_+(\tau)=U_+(\tau)$ for $\tau\in A_+$. Since the restriction of $U_+$ to $A_+$ is increasing, it suffices to show that
\begin{equation}\label{U}
U_+(-1/2,\phi_+(-1/2))=-1/2, \;\;\; U_+(1/2,\phi_+(1/2))=1/2.
\end{equation}
Since $G_2$ and $G_4$ are invariant under the action of $T=\begin{psmallmatrix}1&1\\0&1\end{psmallmatrix}$, we have $g_+(\tau+1)=g_+(\tau)+1$. Then, $V_+(x+1,\phi_+(x))=V_+(x,\phi_+(x))=0$. However, for fixed $x$, there exists a unique $y>1.04$ such that $V_+(x,y)=0$. Thus, we have $\phi_+(x+1)=\phi_+(x)$. From $g_+(\tau+1)=g_+(\tau)+1$, we obtain
\begin{equation}\label{firstUidentity}
U_+\Big(-\frac{1}{2},\phi_+\Big(-\frac{1}{2}\Big)\Big)+1=U_+\Big(\frac{1}{2},\phi_+\Big(\frac{1}{2}\Big)\Big).
\end{equation}
On the other hand, by our choice of the branch of $\sqrt{G_4}$, one can verify that $g_+(-\overline{\tau})=-\overline{g_+(\tau)}$. Substituting $\tau=\frac{1}{2}+\phi_+(\frac{1}{2})i$ and taking real parts, we obtain
\begin{equation}\label{secondUidentity}
U_+\Big(-\frac{1}{2},\phi_+\Big(-\frac{1}{2}\Big)\Big)=U_+\Big(-\frac{1}{2},\phi_+\Big(\frac{1}{2}\Big)\Big)=-U_+\Big(\frac{1}{2},\phi_+\Big(\frac{1}{2}\Big)\Big).
\end{equation}
Combining \cref{firstUidentity} and \cref{secondUidentity}, we obtain \cref{U}.
\end{proof}

By \cref{lem:gplusAplus} and the identity $g_\pm(\tau+1)=g_\pm(\tau)+1$, we obtain the following.

\begin{thm}\label{thm:solution-g}
For each $\alpha\in\RR$, there exists a unique solution
$\tau\in\HH$ of $g_+(\tau)=\alpha$ satisfying $\Img(\tau)>1$.
Moreover, for any $n\in\ZZ$, $\alpha\in[-1/2+n,1/2+n]$ if and only if the unique solution satisfies $\Real(\tau)\in[-1/2+n,1/2+n]$.
The same holds for $g_-$.
\end{thm}

\begin{lem}\label{lem:distinct-zero}
Let $\tau_+$ and $\tau_-$ be solutions of $g_+(\tau)=\alpha$ and $g_-(\tau)=\alpha$, respectively, for some $\alpha\in\RR$, such that $\Img(\tau_+)>1$ and $\Img(\tau_-)>1$. Then $\tau_+\neq\tau_-$.
\end{lem}
\begin{proof}
Suppose that $\tau_0:=\tau_+=\tau_-$. Then
\[
\tau_0-\frac{\pi i}{G_2(\tau_0)+\sqrt{G_4(\tau_0)}}=\tau_0-\frac{\pi i}{G_2(\tau_0)-\sqrt{G_4(\tau_0)}}.
\]
It follows that $G_4(\tau_0)=0$. However, it follows from \cite{RS70} that $G_4(\tau)$ has no zeros in the region $\Img(\tau)>1$, which gives a contradiction.
\end{proof}

By \cref{lem:equiv-G22,thm:solution-g,lem:distinct-zero}, we have the following.

\begin{thm}\label{thm:solution-G22}
For each $g=\begin{pmatrix}
a & b \\ c & d
\end{pmatrix}\in SL_2(\ZZ)$ with $c\neq 0$, the equation $G_{2,2}(g\tau)=0$ has exactly two solutions in $\{\tau\in\HH\mid \Img(\tau)>1\}$.
\end{thm}

Since the interior of the Ford circle corresponding to $a/c$ is the image of
$\{\tau\in\HH\mid \Img(\tau)>1\}$ under $g=\begin{pmatrix}
a & b \\ c & d
\end{pmatrix}\in SL_2(\ZZ)$,
\cref{thm:solution-G22} implies that $G_{2,2}$ has exactly two zeros
inside the Ford circle corresponding to $a/c$.
Finally, by \cref{thm:simple-zero}, we obtain the desired result. This completes the proof of \cref{thm:fordcircle}.



\bibliographystyle{myamsalpha}
\bibliography{ref}

@article{BS10,
author={El Basraoui, Abdellah and Sebbar, Abdellah},
title={Zeros of the Eisenstein series $E_2$},
journal={Proc. Amer. Math. Soc.},
volume={138},
number={7},
year={2010},
pages={2289--2299},
doi={10.1090/S0002-9939-10-10300-1}}

@article{MNS07,
author={Miezaki, Tsuyoshi and Nozaki, Hiroshi and Shigezumi, Junichi},
title={On the zeros of Eisenstein series for $\Gamma_0^*(2)$ and $\Gamma_0^*(3)$},
journal={J. Math. Soc. Japan},
volume={59},
number={3},
year={2007},
pages={693--706},
doi={10.2969/jmsj/05930693}}

@article{BP25,
author={Bhowmik, Tapas and Pathak, Siddhi S.},
title={A note on transcendence of special values of functions related to modularity},
journal={Acta Arithmetica},
volume={218},
year={2025},
pages={1--23},
doi={10.4064/aa240128-7-9}}

@incollection{GKZ06,
author={Gangl, Herbert and Kaneko, Masanobu and Zagier, Don},
title={Double zeta values and modular forms},
booktitle={Automorphic Forms and Zeta Functions},
editor={B{\"o}cherer, Siegfried and Ibukiyama, Tomoyoshi and Kaneko, Masanobu and Sato, Fumihiro},
publisher={World Scientific},
address={Hackensack, NJ},
year={2006},
pages={71--106},
doi={10.1142/9789812774415_0004}}

@article{Gel34,
author={Gelfond, Aleksandr},
title={Sur le septi{\`e}me probl{\`e}me de Hilbert},
journal={Bulletin de l'Acad{\'e}mie des Sciences de l'URSS, Classe des Sciences Math{\'e}matiques et Naturelles},
number={4},
year={1934},
pages={623--634}}

@article{GO22,
author={Gun, Sanoli and Oesterl{\'e}, Joseph},
title={Critical points of Eisenstein series},
journal={Mathematika},
volume={68},
number={1},
year={2022},
pages={259--298},
doi={10.1112/mtk.12124}}

@article{HI17,
author={Hoffman, Michael E. and Ihara, Kentaro},
title={Quasi-shuffle products revisited},
journal={Journal of Algebra},
volume={481},
year={2017},
pages={293--326},
doi={10.1016/j.jalgebra.2017.03.005}}

@article{IJT14,
author={Imamoglu, {\"O}zlem and Jermann, Jonas and T{\'o}th, {\'A}rp{\'a}d},
title={Estimates on the zeros of $E_2$},
journal={Abhandlungen aus dem Mathematischen Seminar der Universit{\"a}t Hamburg},
volume={84},
number={1},
year={2014},
pages={123--138},
doi={10.1007/s12188-014-0091-9}}

@article{Kum21,
author={Kumar, K. Senthil},
title={Linear dependence of quasi-periods over the rationals},
journal={Comptes Rendus. Math{\'e}matique},
volume={359},
number={4},
year={2021},
pages={409--414},
doi={10.5802/crmath.171}}

@article{Kwa20,
author={Kwa{\'s}nicki, Mateusz},
title={A new class of bell-shaped functions},
journal={Transactions of the American Mathematical Society},
volume={373},
number={4},
year={2020},
pages={2255--2280},
doi={10.1090/tran/7825}}

@article{KS22,
author={Kwa{\'s}nicki, Mateusz and Simon, Thomas},
title={Characterisation of the class of bell-shaped functions},
journal={Mathematische Zeitschrift},
volume={301},
number={3},
year={2022},
pages={2659--2683},
doi={10.1007/s00209-022-02997-7}}

@article{Meh13,
author={Meher, Jaban},
title={Some remarks on zeros of quasimodular forms},
journal={Archiv der Mathematik},
volume={101},
year={2013},
pages={121--127},
doi={10.1007/s00013-013-0536-x}}

@article{Nes96,
author={Nesterenko, Yuri V.},
title={Modular functions and transcendence questions},
journal={Sbornik: Mathematics},
volume={187},
number={9},
year={1996},
pages={1319--1348},
note={English translation of Mat. Sb. 187 (1996), no. 9, 65--96},
doi={10.1070/SM1996v187n09ABEH000158}}

@article{RS70,
author={Rankin, F. K. C. and Swinnerton-Dyer, H. P. F.},
title={On the zeros of Eisenstein series},
journal={Bulletin of the London Mathematical Society},
volume={2},
number={2},
year={1970},
pages={169--170},
doi={10.1112/blms/2.2.169}}

@article{Sch37,
author={Schneider, Theodor},
title={Arithmetische Untersuchungen elliptischer Integrale},
journal={Mathematische Annalen},
volume={113},
year={1937},
pages={1--13},
doi={10.1007/BF01571618}}

@article{Sch35,
author={Schneider, Theodor},
title={Transzendenzuntersuchungen periodischer Funktionen. I. Transzendenz von Potenzen},
journal={Journal f{\"u}r die reine und angewandte Mathematik},
volume={172},
year={1935},
pages={65--69},
doi={10.1515/crll.1935.172.65}}

@article{IR24,
author={van Ittersum, Jan-Willem and Ringeling, Berend},
title={Critical points of modular forms},
journal={International Journal of Number Theory},
volume={20},
number={10},
year={2024},
pages={2695--2728},
doi={10.1142/S1793042124501288}}

@article{Sug26,
author={Sugibayashi, Naoki},
title={Critical points of the Eisenstein series for the Fricke group of level 2},
journal={Int. J. Number Theory},
volume={22},
year={2026},
number={3},
pages={489--518},
doi={10.1142/S1793042126500296}}

@article{IR25,
author={van Ittersum, Jan-Willem and Ringeling, Berend},
title={On the zeros of odd weight Eisenstein series},
journal={Mathematika},
volume={71},
year={2025},
number={1},
pages={},
doi={10.1112/mtk.70004}}

@article{IL26,
author={Im, Bo-Hae and Lee, Wonwoong},
title={On the real zeros of depth 1 quasimodular forms},
journal={J. Math. Anal. Appl.},
volume={554},
year={2026},
number={2},
pages={129991},
doi={10.1016/j.jmaa.2025.129991}}

@misc{Bac26,
author        = {Bachmann, Henrik},
title         = {Multiple Zeta Values},
year          = {2026},
eprint        = {2608.02230},
archivePrefix = {arXiv},
primaryClass  = {math.NT},
note          = {arXiv:2608.02230},
}

@article{AOS24,
author={Amdeberhan, Tewodros and Ono, Ken and Singh, Ajit},
title={MacMahon's sums-of-divisors and allied {$q$}-series},
journal={Adv. Math.},
volume={452},
year={2024},
pages={109820},
doi={10.1016/j.aim.2024.109820}}

@article{WY14,
author={Wood, Rachael and Young, Matthew P.},
title={Zeros of the weight two Eisenstein series},
journal={J. Number Theory},
volume={143},
year={2014},
pages={320--333},
doi={10.1016/j.jnt.2014.04.007}}

@article{BG12,
author={Balasubramanian, R. and Gun, Sanoli},
title={On zeros of quasi-modular forms},
journal={J. Number Theory},
volume={132},
year={2012},
pages={2228--2241},
doi={10.1016/j.jnt.2012.04.013}}

@article{SS12,
author={Saber, Hicham and Sebbar, Abdellah},
title={On the critical points of modular forms},
journal={J. Number Theory},
volume={132},
year={2012},
pages={1780--1787},
doi={10.1016/j.jnt.2012.03.004}}

@article{Chu76,
author={Chudnovsky, G. V.},
title={Algebraic independence of constants connected with the exponential and the elliptic functions},
journal={Dokl. Akad. Nauk Ukrain. SSR Ser. A},
number={8},
year={1976},
pages={698--701, 767}}

@book{Sil94,
  author    = {Silverman, Joseph H.},
  title     = {Advanced Topics in the Arithmetic of Elliptic Curves},
  series    = {Graduate Texts in Mathematics},
  volume    = {151},
  publisher = {Springer-Verlag},
  address   = {New York},
  year      = {1994}
}

@article{AGOS25,
author={Amdeberhan, Tewodros and Griffin, Michael and Ono, Ken and Singh, Ajit},
title={Traces of partition Eisenstein series},
journal={Forum Math.},
volume={37},
number={6},
year={2025},
pages={1835--1847},
doi={10.1515/forum-2024-0388}}

@article{AOS25,
author={Amdeberhan, Tewodros and Ono, Ken and Singh, Ajit},
title={Derivatives of theta functions as traces of partition Eisenstein series},
journal={Nagoya Math. J.},
volume={258},
year={2025},
pages={284--295},
doi={10.1017/nmj.2024.30}}

@article {Noz08,
AUTHOR = {Nozaki, Hiroshi},
TITLE = {A separation property of the zeros of {E}isenstein series for
{${\rm SL}(2,\Bbb Z)$}},
JOURNAL = {Bull. Lond. Math. Soc.},
FJOURNAL = {Bulletin of the London Mathematical Society},
VOLUME = {40},
YEAR = {2008},
NUMBER = {1},
PAGES = {26--36},
ISSN = {0024-6093,1469-2120},
MRCLASS = {11F11 (11F03)},
MRNUMBER = {2409175},
MRREVIEWER = {Scott\ Ahlgren},
DOI = {10.1112/blms/bdm117},
URL = {https://doi-org.libproxy.kangwon.ac.kr/10.1112/blms/bdm117},
}

@article {GKPVXZ21,
AUTHOR = {Griffin, Trevor and Kenshur, Nathan and Price, Abigail and
Vandenberg-Daves, Bradshaw and Xue, Hui and Zhu, Daozhou},
TITLE = {Interlacing of zeros of {E}isenstein series},
JOURNAL = {Kyushu J. Math.},
FJOURNAL = {Kyushu Journal of Mathematics},
VOLUME = {75},
YEAR = {2021},
NUMBER = {2},
PAGES = {249--272},
ISSN = {1340-6116,1883-2032},
MRCLASS = {11F11 (11F03)},
MRNUMBER = {4323910},
MRREVIEWER = {Paul\ M.\ Jenkins},
DOI = {10.2206/kyushujm.75.249},
URL = {https://doi-org.libproxy.kangwon.ac.kr/10.2206/kyushujm.75.249},
}

@incollection{KZ95,
author={Kaneko, Masanobu and Zagier, Don},
title={A generalized Jacobi theta function and quasimodular forms},
booktitle={The Moduli Space of Curves (Texel Island, 1994)},
series={Progress in Mathematics},
volume={129},
publisher={Birkh{\"a}user},
address={Boston},
year={1995},
pages={165--172}}

@article{GMR11,
author={Gun, Sanoli and Ram Murty, M. and Rath, Purusottam},
title={Algebraic independence of values of modular forms},
journal={International Journal of Number Theory},
volume={7},
number={4},
year={2011},
pages={1065--1074},
doi={10.1142/S1793042111004769}}

\end{document}